\documentclass[10pt]{amsart}
\usepackage[utf8]{inputenc}
\usepackage[T1]{fontenc}
\usepackage[english]{babel}
\usepackage{color}

\usepackage{comment}
\usepackage{amsmath}
\usepackage{amssymb}
\usepackage{dsfont}
\usepackage{amsfonts}
\usepackage{amsthm}
\usepackage{tikz}
\usepackage{graphicx}
\usepackage{fancyhdr}
\usepackage{caption}
\usepackage{enumerate}
\usepackage{mathtools}

\usepackage[colorlinks=true, pdfstartview=FitH, bookmarks=false]{hyperref}
\hypersetup{urlcolor=black, linkcolor=black, citecolor=black, colorlinks=true}

\newcommand{\vertiii}[1]{{\left\vert\kern-0.25ex\left\vert\kern-0.25ex\left\vert #1 https://www.overleaf.com/project/57d27f441ebc8db50b18ed7d
    \right\vert\kern-0.25ex\right\vert\kern-0.25ex\right\vert}}

\newcommand\N{\mathbb{N}}
\newcommand\Z{\mathbb{Z}}

\newcommand\R{\mathbb{R}}

\newcommand\Hh{\mathcal{H}}

\theoremstyle{plain}
\newtheorem{definition}{Definition}[section]
\newtheorem{proposition}[definition]{Proposition}
\newtheorem{corollary}[definition]{Corollary}
\newtheorem{theorem}[definition]{Theorem}
\newtheorem{lemma}[definition]{Lemma}
\newtheorem*{prop*}{Proposition}
\newtheorem*{lem*}{Lemma}
\newtheorem*{claim}{Claim}

\theoremstyle{remark}
\newtheorem{remark}{Remark}[section]

\newtheorem*{rem*}{Remark}

\usepackage{etoolbox}
\apptocmd{\sloppy}{\hbadness 10000\relax}{}{}
\apptocmd{\sloppy}{\vbadness 10000\relax}{}{}

\title[Local limit theorem for convergent rel. hyperbolic groups]{A local limit theorem for convergent random walks on relatively hyperbolic groups}
\author{Matthieu Dussaule} \thanks{The first author has received funding from the European Research Council  (ERC) under the European Union's Horizon 2020 research and innovation  program under the Grant Agreement No 759702.
}
\author{Marc Peign\'e}
\author{Samuel Tapie}

\date{}

\begin{document}

\begin{abstract}
We study random walks on relatively hyperbolic groups whose law is \emph{convergent},   in the sense that the derivative of its Green function is finite at the spectral radius.
When parabolic subgroups are virtually abelian, we prove that for such a random walk satisfies a local limit theorem of the form $p_n(e,  e)\sim CR^{-n}n^{-d/2}$,  where $p_n(e,  e)$ is the probability of returning to the origin at time $n$,  $R$ is the inverse of the spectral radius of the random walk and $d$ is the minimal rank of a parabolic subgroup along which the random walk is spectrally degenerate.
\end{abstract}

\maketitle

\section{Introduction}
\subsection{General setting}
Consider a finitely generated group $\Gamma$ and a probability measure $\mu$ on $\Gamma$.
The $\mu$-random walk on $\Gamma$ starting at $\gamma \in \Gamma$ is defined as
$$X_n^{\gamma}=\gamma g_1...g_n, $$
where $(g_k)$ are independent random variables of law $\mu$ in $\Gamma$.
The law of $X_n^{\gamma}$ is denoted by $p_n(\gamma, \cdot)$.
For $\gamma=e$,  it is given by the convolution powers $\mu^{*n}$ of the measure $\mu$. The \emph{Local Limit problem} consists in finding {the asymptotic behaviour   of $p_n(e,  e)$ when $n$ goes to infinity.

\medskip

The action by isometries of a discrete group on a Gromov-hyperbolic space $(X,  d)$ is said to be \emph{geometrically finite} if for any $o\in X$,  the accumulation points of $\Gamma o$ on the Gromov boundary $\partial X$ are either \emph{conical limit points} or \emph{bounded parabolic limit points}. We refer to Section~\ref{ssec:relative-hyperbolicity} below for a definition of these notions. A finitely generated group $\Gamma$ is \emph{relatively hyperbolic} with respect to a collection of subgroups $\Omega = \{\Hh_1, ...,  \Hh_p\}$ if it acts via a \emph{geometrically finite action} on a proper geodesic Gromov hyperbolic space $X$,  such that,   up to conjugacy,   $\Omega$ is exactly the set of stabilizers of parabolic limit points for this action. The elements of $\Omega$ are called (maximal) \emph{parabolic subgroups}. We will often assume that parabolic subgroups are virtually abelian.

\medskip

In this paper,  we prove a local limit theorem for a special class of random walks on relatively hyperbolic groups. We will always assume in the sequel that $\mu$ is \emph{admissible},  i.e. its support generates $\Gamma$ as a semigroup,  \emph{symmetric},  i.e. for every $g$, $\mu(g)=\mu(g^{-1})$,  and \emph{aperiodic} i.e. $p_n(e,  e)>0$ for large enough $n$.

On the one hand,  it is known that aperiodic random walks with exponential moments on \emph{virtually abelian groups} of rank $d$ satisfy the following local limit theorem,  see \cite[Theorem~13.12]{Woess-book} and references therein :
\begin{equation}\label{eq:LLT_abelien}
p_n(e,  e)\sim C R^{-n}n^{-d/2}, 
\end{equation}
where $C$ is a positive constant and   $R\geq 1$ is the inverse of the spectral radius of the random walk.

On the other hand,  Gou\"ezel \cite{GouezelLLT} proved that for finitely supported,  aperiodic and symmetric random walks on \emph{non-elementary hyperbolic groups},  the local limit theorem has  always the following form:
\begin{equation}\label{eq:LLT_hyperbolic}
 p_n(e,  e)\sim C R^{-n}n^{-3/2},  
\end{equation}
where,  again,  $C$ is a positive constant and $R$  the inverse of the spectral radius of the random walk. Notice that 
 $R>1$ since non-elementary hyperbolic groups are non-amenable,  see \cite{Kesten}.

\medskip

On \emph{relatively hyperbolic groups},  the first author proved in \cite{DussauleLLT2} that the local limit theorem (\ref{eq:LLT_hyperbolic}) still holds provided the random walk is \emph{spectrally non degenerate}. This will be precisely defined in Definition~\ref{defspectrallydegenerate} below,  see also \cite[Definition~2.3]{DGstability}. Roughly speaking,  the random walk is \emph{spectrally degenerate along a parabolic subgroup $\mathcal H$} if it gives enough weight to $\mathcal H$. When it is spectrally non  degenerate along all the parabolic subgroups in $\Omega$,  its is called \emph{spectrally non degenerate}. The spirit of the result in  \cite{DussauleLLT2} is that a spectrally non degenerate   random walk mainly sees the underlying hyperbolic structure of the group.

In contrast,  for spectrally degenerate random walks,  one would expect to see in the local limit theorem
the appearance of a competition between the exponents $d/2$ and $3/2$,  related to the competition between parabolic subgroups and the underlying hyperbolic structure.

\medskip

The simplest examples of relatively hyperbolic groups are free products. Candellero and Gilch \cite{CandelleroGilch} gave a complete classification of local limit theorems that can occur for nearest neighbour random walks on free products of finitely many abelian groups; in this context,  these free factors play the role of parabolic subgroups.
They indeed proved that whenever the random walk gives enough weight to the free factors,  the local limit theorem is given by (\ref{eq:LLT_abelien}) as in the abelian case,  whereas it is of the form (\ref{eq:LLT_hyperbolic}) in the remaining cases,  see in particular the many examples given in \cite[Section~7]{CandelleroGilch}.

\medskip

Our paper is devoted to the general study of local limit theorems  for the so called \emph{convergent} random walks on a relatively hyperbolic group.  In this case,  the  parabolic subgroups have the maximal possible influence on  the random walk: we make a precise presentation in the next subsection. The main results  of these paper are valid when parabolic subgroups are abelian. Nevertheless,  let us emphasize that a large part of the following  study remains  valid for any convergent random walk, see Remark~\ref{remarkvirtuallyabelianornot} below.

\subsection{Main results}

Let $\mu$ be a probability measure on a relatively hyperbolic group $\Gamma$.
Denote by $R_\mu$ the inverse of its spectral radius,  that is the radius of convergence of the Green function $G(x, y|r)$,  defined as
$$G(x,  y|r)=\sum_{n\geq 0}p_n(x,  y)r^n.$$
This radius of convergence is independent of $x$ and $y$. 

\begin{definition}
Let $\Gamma$ be a relatively hyperbolic group and let $\mu$ be a probability measure on $\Gamma$.
We say that $\mu$,  or equivalently the random walk driven by $\mu$,  is \emph{convergent} if
$$\frac{d}{dr}_{|r=R_\mu}G(e,  e|r)<+\infty.$$
Otherwise,  $\mu$ is said to be \emph{divergent}.
\end{definition}

This terminology was introduced in \cite{DussauleLLT1}. It comes from the strong analogy between \emph{randow walks on relatively hyperbolic groups} on the one hand and \emph{geodesic flow on geometrically finite negatively curved manifolds} on the other hand. We will describe this analogy in Section~\ref{ssec:negative_curvature} below. Spectrally non degenerate random walks on relatively hyperbolic groups are always divergent as shown in \cite[Proposition~5.8]{DussauleLLT1}. All non-elementary cases presented above which show a local limit theorem of abelian type (\ref{eq:LLT_abelien}) come from a convergent random walk.

\medskip

Hence,  if $\mu$ is convergent,  then it is necessarily spectrally degenerate along some parabolic subgroup. Moreover,  whenever parabolic subgroups are virtually abelian,  each of them has a well-defined rank.

\begin{definition}\label{defmaximalrank}
Let $\Gamma$ be a relatively hyperbolic group with respect to virtually abelian subgroups and let $\mu$ be a convergent probability measure on $\Gamma$.
The \emph{rank of spectral degeneracy} of $\mu$ is the minimal rank of a parabolic subgroup along which $\mu$ is spectrally degenerate.
\end{definition}

The central result of this paper is the following local limit theorem.

\begin{theorem}\label{maintheorem}
Let $\Gamma$ be a finitely generated relatively hyperbolic group with respect to virtually abelian subgroups.
Let $\mu$ be a finitely supported,  admissible,  symmetric and convergent probability measure on $\Gamma$.
Assume that the corresponding random walk is aperiodic.
Let $d$ be the rank of spectral degeneracy of $\mu$.
Then for every $x,  y\in \Gamma$ there exists $C_{x,  y}>0$ such that
$$p_n(x,  y)\sim C_{x,  y}R_{\mu}^{-n}n^{-d/2}.$$
If the $\mu$-random walk is not aperiodic,  similar asymptotics hold for $p_{2n}(x,  y)$ if the distance between $x$ and $x'$ is even and for $p_{2n+1}(x,  y)$ if this distance is odd.
\end{theorem}

Note that by \cite[Proposition~6.1]{DussauleLLT1},  the rank of any virtually abelian parabolic subgroup along which $\mu$ is spectrally degenerate is at least $5$. Therefore this local limit theorem cannot coincide with the one given by (\ref{eq:LLT_hyperbolic}) when $\mu$ is spectrally non degenerate.
We get the following corollary.
\begin{corollary}\label{maincorollary}
Let $\Gamma$ be a finitely generated relatively hyperbolic group with respect to virtually abelian subgroups.
Let $\mu$ be a finitely supported,  admissible,  symmetric and convergent probability measure on $\Gamma$ such that the corresponding random walk is aperiodic.
Let $d$ be the rank of spectral degeneracy of $\mu$.
Denote by $q_n(x,  y)$ the probability that the first visit in positive time of  $y$ starting at $x$ is at time $n$.
Then for every $x,  y\in \Gamma$ there exists $C'_{x,  y}>0$ such that
$$q_n(x,  y)\sim C'_{x,  y}R_{\mu}^{-n}n^{-d/2}.$$
\end{corollary}

By the results of Candellero and Gilch \cite[Section~7]{CandelleroGilch},  convergent measures do exist. We do not attempt in this paper  to systematically  construct such measure on any relatively hyperbolic group with virtually abelian parabolic subgroups .

\medskip

It follows from the proof that a local limit theorem analogous to our Theorem~\ref{maintheorem} can be shown (even without assuming virtually abelian parabolic subgroups) as soon as $(\Gamma,  \mu)$ satisfy two conditions:
\begin{itemize}
\item The Martin boundary of $(\Gamma,  \mu)$ is \emph{stable} in the sense of Definition~\ref{def:stable-Martin};
\item The Martin boundary of the first return kernel to any dominant parabolic subgroup is reduced to a point at the spectral radius.
\end{itemize}
We refer to Theorem~\ref{thmcomparingIandJ} for a precise general statement. An important step in our study,  which is of independent interest,  is hence the following fact. 

\begin{theorem}[Theorem~\ref{corostrongstability}]
Let $\Gamma$ be a finitely generated relatively hyperbolic group with respect to virtually abelian subgroups.
Let $\mu$ be a finitely supported,  admissible and symmetric probability measure on $\Gamma$.
Then the Martin boundary of $(\Gamma, \mu)$ is stable.
\end{theorem}

This completes the results of \cite{DGstability}. Along Section~\ref{sectionparabolicGreen},  we prove precise results on the asymptotics of the Green function in a virtually abelian finitely generated group at the spectral radius which will show this stability. 

\begin{remark} \label{remarkvirtuallyabelianornot}
Our proof relies on some important properties satisfied by the virtually abelian parabolic subgroups  $\mathcal H_1,  \ldots,  \mathcal H_p$,  presented in Sections~\ref{sectionparabolicGreen} and \ref{sec:stability}. Once  these properties can be extended   to the case where  the   $\mathcal H_1,  \ldots,  \mathcal H_p$ are \emph{virtually nilpotent parabolic subgroups},  our main result Theorem~\ref{maintheorem} and its consequences may  be extended to  this more general setting,  by using Theorem~\ref{thmcomparingIandJ} and  following the proofs of Section~\ref{sec:symptotics-full-Green}.
\end{remark}

\medskip

For all $k\in \mathbb N$,  we write 
$$
I^{(k)}(r) = \sum_{x_1, ...,  x_k\in \Gamma} G(e,  x_1 |r) G(x_1,  x_2|r)...G(x_{k-1},  x_k|r)G(x_k,  e|r).
$$

It follows from Lemma~\ref{lemmageneralformuladerivatives} that $(\Gamma,  \mu)$ is convergent if and only if $I^{(1)}(R_\mu)<+\infty$. For all parabolic subgroup $\mathcal H < \Gamma$,  we write
$$
I^{(k)}_{\mathcal H}(r) = \sum_{x_1, ...,  x_k\in \mathcal H} G(e,  x_1 |r) G(x_1,  x_2|r)...G(x_{k-1},  x_k|r)G(x_k,  e|r).
$$

The following terminology was introduced in \cite{DussauleLLT1}.

\begin{definition}\label{def:walk-positive-recurrent}
A symmetric admissible and finitely supported random walk $\mu$ on a  relatively hyperbolic group $\Gamma$ is said to be  \emph{spectrally positive recurrent} if:
\begin{enumerate}
\item $\mu$ is divergent,  i.e. $I^{(1)}(R_\mu) = +\infty$;
\item for all parabolic subgroup $\mathcal H<\Gamma$, 
$$
I^{(2)}_{\mathcal H}(R_\mu)<+\infty.
$$
\end{enumerate}
\end{definition}

Any random walk which is spectrally non degenerate is spectrally positive recurrent,  see \cite[Proposition~3.7]{DussauleLLT1}. The terminology \emph{positive  recurrent} is classical for the study of countable Markov shift,  see for instance to \cite{VJ62},   \cite{GS98} and \cite{Sar99}. The analogous of spectral non degeneracy is given for countable Markov shifts by the notion of \emph{strong positive recurrence},  also called \emph{stable positive recurrence},  see \cite{GS98} or \cite{Sar01}. In our setting,  this terminology has been inspired by the close analogy with the study of the geodesic flow on negatively curved manifolds (see Section~\ref{ssec:negative_curvature} below and  \cite[Section~3.3]{DussauleLLT1} for more details).

We will discuss in Section~\ref{sec:positive-recurrent} the relationships between the divergence of the random walk and its spectral positive recurrence. It happens that, when parabolic subgroups  are virtually abelian, both are equivalent unless the random walk is spectrally degenerate along a parabolic subgroup of rank 5 or 6. 
This allows us to classify almost all possible behaviours for $p_n(e,  e)$ on a relatively hyperbolic group whose parabolic subgroups  are virtually abelian,  as illustrated by the following corollary. 

\paragraph{\bf Notation} For two functions $f$ and $g$,  we write $f\lesssim g$ if there exists a constant $C$ such that $f\leq Cg$.
Also write $f\asymp g$ if both $f\lesssim g$ and $g\lesssim f$. If the implicit constants depend on a parameter,  we will avoid this notation.

\begin{corollary}\label{coro:divergent-positive-recurrent}
Let $\Gamma$ be a relatively hyperbolic group whose parabolic subgroups  are virtually abelian. Let $\mu$ be a finitely supported admissible symmetric probability on $\Gamma$. Assume that $\mu$ is spectrally non  degenerate along a parabolic subgroup of rank $5$ or $6$. Then one of the following possibilities occurs.
\begin{description}
	\item [\cite{DussauleLLT1},  Theorem~1.4] If $\mu$ is \emph{spectrally positive recurrent},  then as $n\to +\infty$, 
	$$
	p_n(e,  e)\asymp C R^{-n}n^{-3/2}.
	$$
	\item [\cite{DussauleLLT2},  Theorem~1.1] If $\mu$ is \emph{spectrally non degenerate},  then as $n\to +\infty$, 
	$$
	p_n(e,  e)\sim C R^{-n}n^{-3/2}.
	$$
	\item [Theorem~\ref{maintheorem}] If $\mu$ is \emph{convergent},  then as $n\to +\infty$, 
	$$
	p_n(e,  e)\sim C R^{-n}n^{-d/2}, 
	$$
	where $d$ is the spectral degeneracy rank of $\mu$.
\end{description}

\end{corollary}

We conjecture that for any spectrally positive recurrent walk (even spectrally degenerate), there should be a Local Limit theorem $\displaystyle p_n(e,  e)\sim C R^{-n}n^{-3/2}$. If $\Gamma$ has parabolic subgroups  which are virtually abelian of rank $5$ or $6$, it is possible to construct examples which are divergent but not spectrally positive recurrent, whose local limit theorem is not classified in this corollary.
Examples of such groups with their corresponding local limit theorems  are detailed in a forthcoming article  \cite{DPT23}; see also the remark \ref{remarquefinale} at the end of the present paper.

\subsection{Geodesic flow on negatively curved manifolds and random walks}\label{ssec:negative_curvature}

For convenience of the reader,  we present now some results on the ergodic properties of the geodesic flow on geometrically finite manifolds with negative curvature,  which has influenced this work. 

Let $(M,  g) = (\tilde M,  g)/\Gamma$ be a complete Riemannian manifold,  where $\Gamma = \pi_1(M)$ acts discretely by isometries on the universal cover $(\tilde M,  g)$.
Assume that $M$ has pinched negative curvature,  i.e.\ its sectional curvatures $\kappa_g$ satisfy  the inequalities $-b^2 \leq \kappa_g \leq -a^2 <0$ for some constants $b>a>0$.
Also assume that the action of $\Gamma$ on $(\tilde M,  d_g)$ is \emph{geometrically finite} (see above or Section~\ref{ssec:relative-hyperbolicity} below). By definition $\Gamma$ is hence relatively hyperbolic with respect to a finite family of parabolic subgroups $\mathcal H_1, ...,  \mathcal H_p$,  and the pinched curvature hypothesis implies that the $\mathcal H_k$ are virtually nilpotent. The interested reader will find in \cite{Bowditch95} several other equivalent definitions of geometrical finiteness in the context of smooth negatively curved manifolds. 

For each discrete group $H$ acting by isometries on $(\tilde M,  g)$,  and for all $s>0$,  the \emph{Poincar\'e series} of $H$ at $s$ is
$$
P_H(x,  y | s) = \sum_{\gamma\in H} e^{- s d_g(x, \gamma y)} \in (0,  +\infty].
$$
There is a $\delta_H\geq 0$ independent of $x,  y$ such that this series converges if $s>\delta_H$ and diverges if $s<\delta_H$. This quantity is called the \emph{critical exponent} of $H$. The action of $H$ on $(\tilde M,  g)$ is called \emph{convergent} if $P_H(x,  y|\delta_H)<+\infty$,  and \emph{divergent} otherwise. 

\medskip

Let $\mu$ be a random walk on $\Gamma$. We define for $x,  y\in \Gamma$ the \emph{symmetrized $r$-Green distance} by
\begin{equation}\label{eq:Green-distance}
d_r(x,  y) = \log\left (\frac{G_\mu(x,  y |r)}{G_\mu(e,  e|r)}\right ) + \log\left (\frac{G_\mu(y,  x|r)}{G_\mu(e,  e|r)}\right ).
\end{equation}
This (signed) ``distance'' was introduced in \cite{DussauleLLT1} and is an elaborated version of the classical Green distance defined by Blach\`ere and Brofferio \cite{BlachereBrofferio}.

By Lemma~\ref{lem:r-derivate-Green},   for all $x\in \Gamma$, 
$$
P_\mu(x,  y |r) = \sum_{\gamma \in \Gamma} e^{d_r(x, \gamma y)} = {\frac 1 {G(e,e|r)^2}} \sum_{z\in \Gamma} G(x,  z |r)G(z, x |r) \asymp \frac{d}{dr}_{|r=R_\mu}G(x,  x'|r).
$$

As emphasized by our notation,  the Poincar\'e series $P_\Gamma(x,  y|s)$ of $\Gamma$ is the analogous in this context of group action on a metric space to $P_\mu(x,  y |r)$ for the random walk on $\Gamma$,  which is of same order  as the $r$-derivative of the Green function. The Riemannian metric $g$ has the role of the law $\mu$. The (logarithm of the) critical exponent $\delta_H$ of the Poincar\'e series has the same role as the radius of convergence of the Green function.
The local limit theorem  describes the asymptotic behavior as $n\to +\infty$ of  the quantity  $p^{(n)}(x,  y)$ for any $x,  y \in \Gamma$; in the geometrical setting, it is replaced   by the \emph{orbital counting asymptotic},  that is the asymptotic behavior as $R\to +\infty$ of the orbital function  $N_\Gamma(x,  y,  R)$ defined by :  for all $x,  y\in \tilde M$,  
$$
N_\Gamma(x,  y,  R) := \# \left\{ \gamma\in \Gamma \; ; \; d(x,  \gamma y) \leq R\right\}.
$$

The following definition,  which is the analogous in this context to our previous Definition~\ref{def:walk-positive-recurrent},  comes from the results of \cite{DOP} even though the terminology has been fixed in \cite{PS18} (in the full general setting of negatively curved manifolds,  not necessarily geometrically finite).

\begin{definition}
Let $(M,  g) = (\tilde M,  g)/\Gamma$ be a geometrically finite Riemannian manifold with pinched negative curvature,  where $\Gamma = \pi_1(M)$. Let $o\in \tilde M$ be fixed. The action of $\Gamma$ on $(\tilde M,  g)$ is said to be \emph{positive  recurrent} if
\begin{enumerate}
\item the action of $\Gamma$ on $(\tilde M,  g)$ is divergent;
\item for all parabolic subgroup $\mathcal H\subset \Gamma$, 
$$
\sum_{h\in \mathcal H} d(o,  h.o) e^{-s d(o,  p o)} <+\infty.
$$
\end{enumerate}
\end{definition}

We refer to \cite[Definition~1.3]{PS18} for a definition of positive recurrence beyond geometrically finite manifolds. The action of $\Gamma$ is said to be \emph{strongly positive recurrent}  - in the literature,  one also says that $\gamma$  has  a \emph{critical gap} -  if for all parabolic subgroup $\mathcal H\subset \Gamma$,  we have $\delta_{\mathcal H} < \delta_\Gamma$. This is the analogue of  the notion of \emph{spectrally non degeneracy} for random walks on relatively hyperbolic groups. Theorem~A of \cite{DOP} shows that strongly positive recurrent actions are positive recurrent (only the divergence is non-trivial). This has later been shown for more general negatively curved manifolds in \cite[Theorem~1.7]{ST21} and the analogous result for random walks is given by Proposition~3.7 of \cite{DussauleLLT1}. Moreover,  Theorem~B of \cite{DOP} shows that the action is positive recurrent if and only if the geodesic flow admits an invariant  probability measure of maximal entropy. This has been shown for general negatively curved manifolds in \cite[Theorem~1.4]{PS18}. Combined with Theorem~4.1.1 of \cite{Rob03},  it gives the following asymptotic counting.

\begin{theorem}
Let $(M,  g) = (\tilde M,  g)/\Gamma$ be a negatively curved manifold.
\begin{itemize}
\item If the action of $\Gamma$ is positive recurrent,  then for all $x,  y\in \tilde M$,  there is $C_{xy}>0$ such that,  as $R\to +\infty$, 
$$
N_\Gamma(x,  y,  R) \sim C_{xy} e^{\delta_\Gamma R}.
$$
\item If the action of $\Gamma$ is not  positive  recurrent,  then for all $x,  y\in \tilde M$,  $$N_\Gamma(x,  y) = o\left(e^{\delta_\Gamma R}\right).$$
\end{itemize}

\end{theorem}

Precising the asymptotics of $N_\Gamma(x,  y,  R)$ when the action of $\Gamma$ is not  positive  recurrent is in general difficult. To the authors' knowledge,  the only known examples are abelian coverings (cf \cite{PS94}),  which are not geometrically finite,  and geometrically finite Schottky groups whose parabolic factors have  counting functions satisfying some particular tail condition and for which asymptotics have been obtained in \cite{DPPS},  \cite{Vidotto} and \cite{PTV20}. Recall that Schottky groups are free products of elementary groups whose limit sets are  at a positive distance from each other,  see for instance Section~2.4 of \cite{PTV20} for a definition. In our analogy with random walks on relatively hyperbolic groups,  we emphasize hence the following result from \cite{Peigne},  which gathers  in some particular case results of \cite{Vidotto} and \cite{PTV20}.
It is a Riemannian analogous to the work of Candellero and Gilch \cite{CandelleroGilch} already quoted.

\begin{theorem}\label{thm:phase-transition}
Let $(M,  g_H) = \mathbb H^2/\Gamma$ be a hyperbolic surface where $\Gamma$ is a Schottky group with at least one parabolic free factor $\mathcal H= \langle h\rangle$. We   fix a parameter $b\in (1, 2)$. Then, there exists a family $(g_{a, b})_{a\in (0,  +\infty)}$ of negatively curved Riemannian metrics on $M$ obtained by perturbation of the hyperbolic metric $g_H$ in such a way: 

- the metrics $g_{a, b}$ coincides with  $g_H$  outside a small neighbourhood (controlled by the value $a$) of the  cuspidal end associated with $\mathcal H$;

- the distance $d_{a, b}$ induced by  $g_{a, b}$ satisfies the following condition: for any fixed point $x \in \mathbb H^2$,  
$$
d_{a, b}(x,  p^{n}x)=  2\bigl(\ln \vert n\vert + b \ln\vert \ln \vert n\vert\vert\bigr) + O(1) 
$$
.

Then,  there exists a ``critical value''   $a^*>0$ such that :
\begin{itemize}
\item if $a> a^*$  then the action of $\Gamma$ on $(\mathbb H^2,  g_a)$ is strongly  positive  recurrent. In particular for all $x,  y\in \mathbb H^2$, 
$$
N_\Gamma(x,  y,  R) \sim C_{xy} e^{\delta_\Gamma R};
$$
\item if $a = a^*$,  then the action of $\Gamma$ on $(\mathbb H^2,  g_a)$ is divergent but non positive  recurrent. Moreover,  for all $x,  y\in \mathbb H^2$, 
$$
N_\Gamma(x,  y,  R) \sim C_{xy} \frac{e^{\delta_\Gamma R}}{R^{2-b}};
$$

\item if $a\in (0,  a^*)$,  then the action of $\Gamma$ on $(\mathbb H^2,  g_a)$ is convergent and for all $x,  y\in \mathbb H^2$, 
$$
N_\Gamma(x,  y,  R) \sim C_{xy} \frac{e^{\delta_\Gamma R}}{R^b}.
$$
\end{itemize}
\end{theorem}

It follows from Corollary~\ref{coro:divergent-positive-recurrent} that the  critical situation  of  a divergent but non  positive  recurrent action   which appears above in the   context of geometrically finite negatively curved surfaces when $a = a^*$    can  occur in the context of random walks on relatively hyperbolic groups with virtually abelian parabolic subgroups in the only case when the random walk is spectrally degenerate along some parabolic subgroups of rank $5$ or $6$.

We end this paragraph with a table that summarizes the different cases that arise in the study of local limit theorems of relatively hyperbolic groups. We also indicate the corresponding results obtained in the framework of geometrically finite non-compact surfaces endowed  with the metric $g_a$ defined in Theorem~\ref{thm:phase-transition}. 


\begin{center}
\begin{tabular}{|c|c|}
\hline
{\bf Local Limit Theorem}   & {\bf Counting problem} \\
      \hline
      {\bf $\Gamma$ spectrally non degenerate}& {\bf critical gap property} $\delta_\Gamma>\delta_\mathcal H$
      \\
       ({\it $\mu$ is spectrally positive recurrent})& ({\it  $\Gamma$ is positive recurrent})
      \\
      \quad & \quad \\
   $p_n(x,  y)\sim C_{x, y} R_\mu^{-n}n^{-3/2}$  &   $N_\Gamma(x,  y,  R) \sim C_{xy}  e^{\delta_\Gamma R}$ 
   \\
     \quad & \quad \\
 see \cite{DussauleLLT2}  &
 see  \cite{Rob03}, \cite{Peigne}  \\
   \hline
   {\bf $\mu$ spectrally  degenerate }& $\Gamma$ {\bf exotic} i.e $\  \delta_\Gamma =\delta_\mathcal H$
      \\
      +&+\\
      {\bf  spectrally positive recurrent }& {\bf positive recurrent}       \\
      \  & \  \\ 
      \  & \ \\
\underline{Rough estimate}:    $p_n(x,  y)\asymp  R_\mu^{-n}n^{-3/2}$  &   $N_\Gamma(x,  y,  R) \sim C_{xy}  e^{\delta_\Gamma R}$ 
   \\
     \  &  \  \\
 see \cite{DussauleLLT1}&
 see \cite{Rob03} \cite{Peigne} \\
     \  &  \  \\
 \underline{Conjecture}:    $p_n(x,  y)\sim C_{x, y}  R_\mu^{-n}n^{-3/2}$  &  \  \\
      \  &  \  \\
   \hline
{\bf $\mu$ spectrally  degenerate }& $\Gamma$ {\bf exotic} i.e $\  \delta_\Gamma =\delta_\mathcal H$
      \\
      + {\bf divergent} + &+ {\bf divergent} +\\
      {\bf not   spectrally  positive recurrent}& {\bf not positive recurrent}       \\
      \quad &\  \\
   degeneracy rank $5$ or $6$   &   $N_\Gamma(x,  y,  R) \sim C_{xy}   \frac{e^{\delta_\Gamma R}}{R^{2-b}}$ 
   \\
   possible exotic Local limit theorem & \ \\
     \quad & \quad \\
 see \cite{DPT23}&
 see \cite{Vidotto}, \cite{Peigne} \\
   \hline
  
       {\bf $\mu$ convergent}   & $\Gamma$ {\bf convergent}  \\
   ({\it thus $\mu$ spectrally degenerate})& ({\it thus $\Gamma$  exotic)}       \\
      \quad &\  \\
 $p_n(x,  y)\sim C_{x,  y}R_{\mu}^{-n}n^{-d_\mu/2} $ &   $N_\Gamma(x,  y,  R) \sim C_{xy} \frac{e^{\delta_\Gamma R}}{R^b}$ 
   \\
     \quad & \quad \\
 see Theorem~\ref{maintheorem}&
 see \cite{PTV20} \\
   \hline

   \end{tabular}
   \end{center}

On the left column, $\Gamma$ is a  relatively hyperbolic group with respect to virtually abelian parabolic subgroups
 $\mathcal H_1,$ \ldots, $\mathcal H_p$ (up to conjugacy). We consider an aperiodic $\mu$-random walk on $\Gamma$ where   $\mu$  is  a finitely supported,  admissible and  symmetric  probability measure on $\Gamma$ with rank $d_\mu$. On the right column, $\Gamma$ is a  geometrically finite Fuchsian group  with parabolic subgroups 
 $\mathcal H_1, \ldots, \mathcal H_p$ (up to conjugacy).  We assume that $\mathbb H^2/\Gamma$ is endowed with the metric $g_{a, b}$ and set  $\delta_{\mathcal H}= \max(\delta_{\mathcal H_1}, \ldots, \delta_{\mathcal H_p})$.
\subsection{Organization of the paper}

In Section~\ref{sec:basics},  we present our setting: relatively hyperbolic groups,  transition kernels with their Green function and their Martin boundary,  and Ancona inequalities which roughly state that the random walk tracks relative geodesics with high probability in a relatively hyperbolic group.

In Section~\ref{sectionparabolicGreen},  we introduce the \emph{first return kernel} $p_{\mathcal H}(., .|r)$ to a parabolic subgroup $\mathcal H$ of rank $d$. Assuming that the random walk is spectrally degenerate along $\mathcal H$ we give asymptotics for the $j$-th derivative of the Green function associated to $p_{\mathcal H}$ where $j=\lceil d/2\rceil -1$ (see Proposition~\ref{propestimatesderivativesparabolicGreen}).

In Section~\ref{sec:stability},  we assume that  parabolic subgroups are virtually abelian and show that  the Martin boundary is stable in the sense of Definition~\ref{def:stable-Martin} below (see Theorem~\ref{corostrongstability}). This had already been shown in \cite{DGstability} when the random walk is spectrally non degenerate; it still holds  for convergent (hence spectrally degenerate) random walks.

In Section~\ref{sec:symptotics-full-Green},  we assume that the Martin boundary of the full random walk is stable and the Martin boundary of the walk restricted to parabolic subgroups is reduced to a point. Under these conditions,  we prove that asymptotics for the $j$-ith derivative of the Green function of the full random walk are given by the analogous asymptotics for the transition kernels of the first return  to the parabolic subgroups along which the walk is spectrally degenerate (see Theorem~\ref{thmcomparingIandJ}).

In Section~\ref{sec:proofLLT},  we gather the ingredients of the three previous section which give asymptotics for the $j$-th derivative of the full Green function,  where $j=\lceil d/2\rceil -1$ and $d$ is the rank of spectral degeneracy of the walk.
Theorem~\ref{maintheorem} follows,    applying  a Tauberian type theorem shown in \cite{GouezelLalley}.

Eventually,  in Section~\ref{sec:positive-recurrent},  we show that,  if the parabolic subgroups are virtually abelian and  the random walk is divergent,  then the random walk is automatically spectrally positive recurrent. 

\section{Random walks on relatively hyperbolic groups}\label{sec:basics}

\subsection{Relatively hyperbolic groups and relative automaticity}\label{ssec:relative-hyperbolicity}

\subsubsection{Limit set}
Consider a discrete group $\Gamma$ acting by isometries on a Gromov-hyperbolic space $X$. Let $o\in X$ be fixed. Define the limit set $\Lambda \Gamma$ as the adherence of $\Gamma o$ in the Gromov boundary $\partial X$ of $X$. This set does not depend on $o$.

A point $\xi\in \Lambda \Gamma$ is called \emph{conical} if there is a sequence $(\gamma_{n})_n$ in  $\Gamma$ and distinct points $\xi_1, \xi_2$ in $\Lambda \Gamma$ such that,  for  all $\xi\neq \zeta$ in $\Lambda \Gamma$,  the sequences 
$(\gamma_{n}\xi)_n$  and $(\gamma_{n}\zeta)_n$ converge to $\xi_1$ and  $\xi_2$ respectively.
A point $\xi\in \Lambda \Gamma$ is called \emph{parabolic} if its stabilizer $\Gamma_\xi$  in $\Gamma$ is infinite  and the  elements of $\Gamma_\xi$ fix  only $\xi$ in $\Lambda\Gamma$.
A parabolic limit point $\xi$ in $\Lambda \Gamma$ is said  \emph{bounded} if $\Gamma_\xi$ acts cocompactly on $\Lambda \Gamma \setminus \{\xi\}$. The action of $\Gamma$ on $X$ is said to be \emph{geometrically finite} if $\Lambda \Gamma$  only contains conical limit points and bounded parabolic limit points.

\subsubsection{Relatively hyperbolic groups}

There are in the literature several equivalent definitions of relatively hyperbolic groups. Let us present the two characterizations which we will use in this paper. We refer to \cite{Bowditch},  \cite{Farb} and \cite{Osin} for more details.

Let $\Gamma$ be a finitely generated groups and $S$ be a fixed generating set. 
Let $\Omega_0$ be a finite collection of subgroups,  none of them being conjugate. Let $\Omega$ be the closure of $\Omega_0$ under conjugacy.

The \emph{relative graph} $\hat \Gamma = \hat\Gamma(S,  \Omega_0)$  is the Cayley graph of $\Gamma$ with respect to $S$ and the union of all $\mathcal{H}\in \Omega_0$ \cite{Osin}.
It is quasi-isometric to the coned-off graph introduced by Farb in \cite{Farb}.
The distance $\hat{d}$ in $\hat \Gamma$ is  called  the \emph{relative distance}. We also denote by $\hat{S}_n$ the sphere of radius $n$ centered at $e$ in $\hat{\Gamma}$. Eventually,  we will call \emph{relative geodesic} a geodesic in $\hat{\Gamma}$.

\begin{theorem}[\cite{Bowditch}]
Using the previous notations,  the following conditions are equivalent.

\begin{enumerate}
\item The group $\Gamma$ has a geometrically finite action on a Gromov hyperbolic space $X$ such that the parabolic limit points are exactly the fixed points of elements in $\Omega$.

\item The relative graph $\hat \Gamma(S,  \Omega_0)$ is Gromov hyperbolic for the relative distance $\hat d$,  and for all $L>0$ and all $x\in \hat \Gamma$,  there exists finitely  many  closed loop of length $L>0$ which contains $x$.
\end{enumerate}

When these conditions are satisfied,  the group $\Gamma$ is said to be \emph{relatively hyperbolic} with respect to $\Omega_0$.
\end{theorem}

Assume now that $\Gamma$ is relatively hyperbolic with respect to $\Omega$,  and let $X$ be a Gromov hyperbolic space on which $\Gamma$ has a geometrically finite action whose parabolic subgroups  are the element of $\Omega$. The limit set $\Lambda \Gamma\subset \partial X$ is called the \emph{Bowditch boundary} of $\Gamma$. It is unique up to equivariant homeomorphism.

\medskip

Archetypal examples of relatively hyperbolic groups with respect to virtually abelian subgroups are given by finite co-volume Kleinian groups.
In this case,  the group acts via a geometrically finite action on the hyperbolic space $\mathbb{H}^n$ and the Bowditch boundary is the full sphere at infinity $\mathbb{S}^{n-1}$.

\subsubsection{Automatic structure}

The notion of relative automaticity was introduced by the first author in \cite{DussauleLLT1}. 

\begin{definition}\label{definitionautomaticstructure}
A \emph{relative automatic structure}  - or shortly an \emph{automaton} -   for $\Gamma$ with respect to the collection of subgroups $\Omega_0$ and with respect to some finite generating set $S$ is a directed graph $\mathcal{G}=(V,  E,  v_*)$ with a distinguished vertex $v_*$ called the \emph{starting vertex},  where the set of vertices $V$ is finite and with a labelling map $\phi:E\rightarrow S\cup \bigcup_{\mathcal{H}\in \Omega_0}\mathcal{H}$ such that the following holds.
If $\omega=e_1, \ldots,  e_n$ is a path of adjacent edges in $\mathcal{G}$,  define $\phi(e_1, ...,  e_n)=\phi(e_1)...\phi(e_n)\in \Gamma$.
Then, 
\begin{itemize}
    \item no edge ends at $v_*$,  except the trivial edge starting and ending at $v_*$, 
    \item every vertex $v\in V$ can be reached from $v_*$ in $\mathcal{G}$, 
    \item for every path $\omega=e_1, ...,  e_n$,  the path $e, \phi(e_1), \phi(e_1e_2), ..., \phi(\gamma)$ in $\Gamma$ is a relative geodesic from $e$ to $\phi(\gamma)$,  i.e.\ the image of $e, \phi(e_1), \phi(e_1e_2), ..., \phi(\gamma)$ in $\hat{\Gamma}$ is a geodesic for the metric $\hat{d}$, 
    \item the extended map $\phi$ is a bijection between paths in $\mathcal{G}$ starting at $v_*$ and elements of $\Gamma$.
\end{itemize}
\end{definition}


\begin{theorem}\label{codingrelativegeodesics}\cite[Theorem~4.2]{DussauleLLT1}
If $\Gamma$ is relatively hyperbolic with respect to $\Omega$,  then for any finite generating set $S$ and for any choice of a full family $\Omega_0$ of representatives of conjugacy classes of elements of $\Omega$,  $\Gamma$ is relatively automatic with respect to $S$ and $\Omega_0$.
\end{theorem}

This statement  is proved by first constructing an automaton that encodes relative geodesics,  then  showing that there exist finitely many  relative cone-types,  see \cite[Definition~4.7,  Proposition~4.9]{DussauleLLT1} for more details.
To obtain a bijection between paths in the automaton and elements of $\Gamma$,  one fixes an  order on the union of $S$ and all the $\mathcal{H} \in \Omega_0$,  which allows to  choose the smallest possible relative geodesics for the associated lexicographical order.

Relative automaticity is a key point in \cite{DussauleLLT2} to prove a local limit theorem in the spectrally non degenerate case.
It is again of crucial use in this paper, see Section~\ref{sec:symptotics-full-Green}.


\subsection{Transition kernels and Martin boundaries}

 Let us make a general presentation of the notion of Martin boundaries. 
In what follows,   $E$ is  a countable space endowed  with the discrete topology and $o$ is a fixed base point in $E$

\begin{definition}
A \emph{transition kernel} $p$ on $E$ is a positive map
$p:E\times E\rightarrow \R_+$ with \emph{finite total mass},  i.e. such that 
$$\forall x\in E,  \sum_{y\in E}p(x,  y)<+\infty.$$
\end{definition}
When the total mass is 1 (which we will not require),  we call it a \emph{probability transition kernel}. It then defines a \emph{Markov chain on $E$},  i.e. a random process $(X_n)_{n\geq 0}$ such that $P(X_{n+1} = b | X_n = a) = p(a,  b)$. In general,  we will say that $p$ defines a \emph{chain} on $E$.

If $\mu$ is a probability measure on a finitely generated group $\Gamma$,  then the kernel $p_\mu(g,  h)=\mu(g^{-1}h)$ is a probability transition kernel and the corresponding Markov chain is  the $\mu$-random walk.

\begin{definition}
Let $p : E\times E \to \R_+$ be a transition kernel on $E$.
\begin{itemize}
\item The \emph{Green function} associated to $p$ is defined by
$$G_p(x,  y)=\sum_{n\geq 0}p^{(n)}(x,  y)\in [0, +\infty], $$
where $p^{(n)}$ is the $n$th convolution power of $p$,  i.e.\
$$p^{(n)}(x,  y)=\sum_{x_1, \ldots,  x_{n-1}\in E}p(x,  x'_1)p(x_1,  v_2)\cdots p(x_{n-1},  y).$$

\item The chain defined by $p$ is \emph{finitely supported} if for every $x\in E$,  the set of $y\in E$ such that $p(x,  y)>0$ is finite.

\item The chain is \emph{admissible} (or \emph{irreducible}) if
for every $x,  y\in E$,  there exists $n$ such that $p^{(n)}(x,  y)>0$.

\item The chain is \emph{aperiodic} (or \emph{strongly irreducible}) if
for every $x,  y\in E$,  there exists $n_0$ such that $\forall n\geq n_0$,  $p^{(n)}(x,  y)>0$.

\item The chain is \emph{transient} if the Green function is everywhere finite.
\end{itemize}
\end{definition}

Consider a transition kernel $p$ defining an irreducible transient chain. For $y\in E$,  define the Martin kernel based at $y$ as
$$K_p(x,  y)=\frac{G_p(x,  y)}{G_p(o,  y)}.$$
The \emph{Martin compactification} of $E$ with respect to $p$ and $o$ is a compact space containing $E$ as an open and dense space, 
whose topology is described as follows. A sequence $(y_n)_n$ in $E$ converges to a point $\xi$ in the Martin compactification if and only if
the sequence $(K(\cdot,  y_n))_n$ converges pointwise to a function which we write $K(\cdot,  \xi)$.
Up to isomorphism,  it does not depend on the base point $o$ and we denote it by $\overline{E}_p$.
We also define the \emph{$p$-Martin boundary} (or \emph{Martin boundary},  when there is no ambiguity) as $\partial_pE=\overline{E}_p\setminus E$.
We refer for instance to \cite{Sawyer} for a complete construction of the Martin compactification.

The Martin boundary contains a lot of information.
It was first introduced to study non-negative harmonic functions.
We will use it here to prove our local limit theorem.

\medskip
Let us now define the notion of stability for the Martin boundary,  following Picardello and Woess \cite{PicardelloWoess}.
Assuming that $p$ is irreducible,  the radius of convergence of the Green function $G_p(x,  y)$ is independent of $x$ and $y$.
Denote it by $R_p$ and for all $0 \leq r \leq R_\mu$ let us  set $G_p(x,  y|r)=G_{rp}(x,  y)$,  i.e.
$$G_p(x,  y|r)=\sum_{n\geq0}r^np^{(n)}(x,  y).$$
Also set $K(x,  y|r)=K_{rp}(x,  y)$.
The Martin compactification,  respectively the boundary associated with $K(\cdot, \cdot|r)$,  is called the $r$-Martin compactification,  respectively the $r$-Martin boundary,  and is denoted by $\overline{E}_{rp}$,  respectively by $\partial_{rp}E$.

\begin{definition}\label{def:stable-Martin}
The Martin boundary of $E$ with respect to $p$ is \emph{stable} if the following conditions hold.
\begin{enumerate}
    \item For every $x,  y\in E$,  we have $G_p(x,  y |R_p) <+\infty$ where $R_p$ is the radius of convergence of the Green function.
    \item For every $0<r_1,  r_2< R_p$,  the sequence $(K(\cdot,  y_n|r_1))_n$ converges pointwise if and only if $(K(\cdot,  y_n|r_2))_n$ converges pointwise,  i.e.\ the $r_1$ and $r_2$-Martin compactifications are homeomorphic.
    For simplicity we then write $\partial_{p}\Gamma$ for the $r$-Martin boundary whenever $0<r<R_p$.
    \item The identity on $\Gamma$ extends to a continuous and equivariant surjective map $\phi_{p}$ from $E\cup \partial_{p}E$ to $E \cup \partial_{R_pp}E$.
    We then write $K(x, \xi|R_p)=K(x, \phi_{p}(\xi)|R_p)$ for $\xi\in \partial_pE$.
    \item The map $(x, \xi,  r)\in E \times \partial_pE\times (0,  R_p] \mapsto K(x, \xi|r)$ is continuous with respect to  $(x, \xi,  r)$.
\end{enumerate}
We say that the Martin boundary is \emph{strongly stable} if it is stable and the second condition holds for every $0<r_1,  r_2\leq R_p$; in this case,   the map $\phi_{p}$ induces a homeomorphism from the $r$-Martin boundary to the $R_p$-Martin boundary.
\end{definition}

If $p$ is the transition kernel of an admissible random walk on a finitely generated group which is \emph{non-amenable},  it has been shown by Guivarc'h in \cite[p.~20,  remark~b]{Guivarch} that the condition (1) is always satisfied. Note that non-elementary relatively hyperbolic groups are always non-amenable. 

The Martin boundary of any finitely supported symmetric admissible random walk on a \emph{hyperbolic group} is strongly stable. More generally,  the Martin boundary of a finitely supported symmetric and admissible random walk on a \emph{relatively hyperbolic group} is studied in \cite{GGPY},  \cite{DGGP} and \cite{DGstability}. In particular,  whenever the parabolic subgroups are virtually abelian,  the homeomorphism type of the $r$-Martin boundary is described in \cite{DGstability}. It is proved there that  the Martin boundary is strongly stable if and only if the random walk is  spectrally non  non  degenerate.
We will prove in Section~\ref{sec:stability} that stability  (but not strong stability)  still holds in the spectrally degenerate case.

\medskip

Let us eventually mention a central computation which we  use many times.

\begin{lemma}\label{lem:r-derivate-Green}
Let $p : E\times E \to \R_+$ be a transition kernel and for all $r\in [0,  R_p]$,  write again $G_p(x,  y|r)=\sum_{n\geq0}r^np^{(n)}(x,  y).$ Then
$$
\frac d {dr} \left( r G_p(x,  y |r)\right) = \sum_{z\in E} G(x,  z |r) G(z,  y |r).
$$
\end{lemma}

We refer to \cite[Lemma~3.1]{DussauleLLT1} for a proof,  which is a standard manipulation of power series. The generalization to higher derivatives is given by Lemma~\ref{lemmageneralformuladerivatives}.

\medskip

{\bf Standing assumptions.}
From now,  and until the end of this paper,  we fix a finitely generated group $\Gamma$ relatively hyperbolic with respect to a finite collection of parabolic subgroups $\Omega_0 = \{\mathcal H_1,  ...,  \mathcal H_N\}$. We fix a finitely supported symmetric probability measure $\mu$ on $\Gamma$ whose associated random walk is admissible and irreducible. Eventually,  we assume that  the support  $S$ of $\mu$ is a generating set,  which is fixed from now on; the distance on $\Gamma$ is the word distance induced by $S$. This implies in particular that for all $x,  y\in \Gamma$,  if $x$ and $y$ are on the same geodesic in $\Gamma$,  then there exists $n>0$ such that $p^{(n)}(x,  y) >0$. 

In the sequel,  we  denote  by $\hat \Gamma$    the  relative graph $  \hat \Gamma(S,  \Omega_0)$,   by $R_\mu$ the inverse of the spectral radius of $\mu$ and by $G(x,  y|r)$ the Green function,  where $0\leq r\leq R_\mu$ and $x,  y\in \Gamma$. As already mentioned,  since $\Gamma$ is not amenable and $\mu$ is admissible,  it follows from \cite{Guivarch} that   $G(x,  y | R_\mu) <+\infty$ for all $x,  y\in \Gamma$.


\subsection{Relative Ancona inequalities}

For any set $A\subset \Gamma$,  we set 
\begin{equation}\label{eq:relative-Green}
G(x,  y;A|r):= \sum_{n\geq 1}\sum_{g_1, ...,  g_{n-1}\in A}r^n\mu(h^{-1}g_1)\mu(g_1^{-1}g_2)...\mu(g_{n-2}^{-1}g_{n-1})\mu(g_{n-1}^{-1}h'); 
\end{equation}
this quantity is called the  \emph{relative Green function} of paths staying in $A$ except maybe at their beginning and end. Writing $A^c = \Gamma \backslash A$,  the relative Green function $p_{A,  r}(., .) := G(., .;A^c |r)$ is  called   the \emph{first return kernel} to $A$. 

For all $y\in \Gamma$ and $\eta>0$, we write 
$$B_\eta(y) = \{z\in \Gamma \mid  \hat d(y,z)  \leq \eta\}.$$ 
We will use repeatedly the following results.

\begin{proposition}\label{Anconaavoidingball}\cite[Corollary~3.7]{DGstability}
For every $\epsilon>0$ and every $R\geq 0$,  there exists $\eta$ such that the following holds.
 For every $x,  y,  z$  such that $y$ is within $R$ of a point on a relative geodesic from $x$ to $z$ and  for every $r\leq R_\mu$,
$$G(x,  z;B_\eta(y)^c|r)\leq \epsilon G(x,  z|r).$$
\end{proposition}

This proposition can be interpreted as follows : with high probability,  a random path from $x$ to $z$ has to pass through a neighborhood of $y$,  whenever $y$ is on a relative geodesic from $x$ to $z$.
As a consequence,  we have the following.

\begin{proposition}\label{weakAncona}(Weak relative Ancona inequalities)
For every $R\geq 0$,  there exists $C$ such that the following holds.
For every $x,  y,  z$  such that $y$ is within $R$ of a point on a relative geodesic from $x$ to $z$ and  for every $r\leq R_\mu$,
$$\frac{1}{C}G(x,  y|r)G(y,  z|r)\leq G(x,  z|r)\leq C G(x,  y|r)G(y,  z|r).$$
\end{proposition}

This is proved by decomposing a trajectory from $x$ to $z$ according to its potential first visit to $B_\eta(y)$,  where $\eta$ is chosen such that $G(x,  z;B_\eta(y)^c|r)\leq 1/2 G(x,  z|r)$ from Proposition~\ref{Anconaavoidingball},  see precisely \cite[Theorem~5.1,  Theorem~5.2]{GGPY}.


These inequalities were first proved by Ancona \cite{Ancona} in the context of hyperbolic groups for $r=1$.
The uniform inequalities up to the spectral radius were proved by Gouezel and Lalley \cite{GouezelLalley} for co-compact Fuchsian groups and then by 
Gouezel \cite{GouezelLLT} in general.
For relatively hyperbolic groups,  Gekhtman,  Gerasimov,  Potyagailo and Yang \cite{GGPY} proved them for $r=1$.
The uniform inequalities up to the spectral radius were then proved by the first author and Gekhtman \cite{DGstability}.
They play a key role in the identification of the Martin boundary,  see \cite{GGPY} for several other applications.

Let us mention that there exist \emph{strong relative Ancona inequalities} (cf \cite[Definition~2.14]{DussauleLLT2}),  that are a key ingredients in \cite{DussauleLLT2} to prove a local limit theorem in the spectrally non degenerate case.
However,  we do not need  them in the present paper.



\section{Asymptotics of the first return to parabolic Green functions}\label{sectionparabolicGreen}

Throughout  this section,  we fix  a parabolic subgroup $\mathcal{H}\in \Omega_0$. For $\eta\geq 0$,  the $\eta$-neighbourhood of $\mathcal{H}$ is noted   $\mathcal{N}_\eta(\mathcal{H})$ . 

We   introduce below the \emph{first return transition kernel} to  $\mathcal{N}_\eta(\mathcal{H})$. The main goal of this section is showing asymptotics for the derivatives of the Green function associated to this first return kernel.

\subsection{First return transition kernel and spectral degeneracy}

 For $r\leq R_\mu$,   let $p_{\mathcal H, \eta,  r}(h,  h') = G(h,  h';\mathcal{N}_\eta(\mathcal{H})^c|r)$ be the \emph{first return kernel} to $\mathcal{N}_\eta(\mathcal{H})$,  i.e. 
$$p_{\mathcal H, \eta,  r}(h,  h')=\sum_{n\geq 1}\sum_{\underset{\notin \mathcal{N}_\eta(\mathcal{H})}{g_1, ...,  g_{n-1}}}r^n\mu(h^{-1}g_1)\mu(g_1^{-1}g_2)...\mu(g_{n-2}^{-1}g_{n-1})\mu(g_{n-1}^{-1}h').$$

For simplicity,  when $\eta=0$,  we  write $p_{\mathcal H,  r}=p_{\mathcal H, 0,  r}$.

The $n^{th}$-convolution power  of $p_{\mathcal H,  \eta,  r}$ is noted $p_{\mathcal H, \eta,  r}^{(n)}$  and    the associated Green function,  evaluated at $t$,  is  
$$
G_{\mathcal H, \eta,  r}(h,  h'|t) := \sum_{n\geq 1} p_{\mathcal H, \eta,  r}^{(n)}(h,  h') t^n.
$$
The radius of convergence $R_{\mathcal H, \eta}(r)$ of this power series  is the inverse of the spectral radius of the associated chain. 

For simplicity,  we write $R_{\mathcal H, \eta}=R_{\mathcal H, \eta}(R_{\mu})$ and $R_{\mathcal H}=R_{\mathcal H, 0}(R_\mu)$.
Recall the following definition from \cite{DGstability}.

\begin{definition}\label{defspectrallydegenerate}
The measure $\mu$,  or equivalently the random walk,  is said to be \emph{spectrally degenerate along $\mathcal{H}$} if $R_{\mathcal H}=1$.
\end{definition}

Since $\mathcal H$ is fixed for the remainder of the section,  we drop the index $\mathcal H$ in the notations.
We now enumerate a list of properties satisfied by  $p_{\eta,  r}$ and $G_{\eta,  r}$.
\begin{itemize}
\item Since the  $\mu$-random walk on $\Gamma$ is invariant under the action of $\Gamma$,  the kernel $p_{\eta,  r}$ is $\mathcal{H}$-invariant:  in other words,  $p_{\eta,  r}(hx,  hy)=p_{\eta,  r}(x,  y)$ for any $h\in \mathcal{H}$ and any $x,  y\in \mathcal{N}_\eta(\mathcal{H})$.

\item By definition,  the first return transition kernel satisfies $p_{\eta,  r}(x,  y)>0$ if and only if there is a \emph{first return path} in $\mathcal{N}_\eta(\mathcal{H})$ from $x$ to $y$,  i.e. if there exists $n\geq 0$ and a path $x,  g_1, ...,  g_n,  y$ in $\Gamma$ with positive probability  such that $g_i\notin\mathcal{N}_\eta(\mathcal{H})$ for $i=1, ...,  n$.

\item  The chain with kernel   $p_{\eta,  r}$ is admissible,  i.e.\ for every $x,  y\in \Gamma$,  there exists $n$ such that $p_{\eta,  r}^{(n)}(x,  y)>0$,  see \cite[Lemma~5.9]{DGGP} for a complete proof.

\end{itemize}

The following lemma shows that when $x$ and $y$ are in $\mathcal{N}_\eta(\mathcal{H})$,  the relative Green function equals the full Green fucntion. The proof is straightforward (see \cite[Lemma~4.4]{DGstability}). It will be frequently used.

\begin{lemma}\label{lem:relative-Green=Green}
Under the previous notations,  for all $x,  y\in \mathcal{N}_\eta(\mathcal{H})$,  any $r\leq R_\mu$ and any $\eta\geq 0$, 
$$G_{\eta,  r}(x,  y|1)=G(x,  y|r), $$
\end{lemma}

For $x,  y\in \Gamma$,  we write $d_{\mathcal{H}}(x,  y)$ the distance between the projections $\pi_{\mathcal{H}}(x)$ and $\pi_{\mathcal{H}}(y)$ of $x$ and $y$ respectively onto $\mathcal{H}$.
Since projections on parabolic subgroups are well-defined up to a uniformly bounded error term,  see \cite[Lemma~1.15] {Sisto-projections},  $d_{\mathcal{H}}(x,  y)$ is also defined up to a uniformly bounded error term.
Letting $M\geq 0$,  we say that $p_{\eta,  r}$ has \emph{exponential moments up to $M$} if for any $x\in \mathcal{N}_\eta(\mathcal{H})$,  we have
$$\sum_{y\in \mathcal{N}_\eta(\mathcal{H})} p_{\eta,  r}(x,  y)\mathrm{e}^{Md_{\mathcal{H}}(x,  y)}<+\infty.$$

The following lemma is the main reason for introducing $p_{\mathcal H, \eta, r}$ with $\eta>0$.
\begin{lemma}\cite[Lemma~4.6]{DGstability}\label{exponentialmoments}
For every $M\geq 0$,  then exists $\eta_M$ such that for every $\eta \geq \eta_M$ and for every $r\leq R_\mu$,  the kernel $p_{\eta,  r}$ has exponential moments up to $M$.
\end{lemma}

Note that $\eta_M$ does not depend on $r$; hence,  choosing the neighborhood of $\mathcal{H}$ large enough,  the kernels $p_{\eta,  r}$ have exponential moments up to $M$,  uniformly in $r$.


\subsection{Vertical displacement transition matrix}\label{ssec:vertical-displacement-matrix}

Until the end of Section~\ref{sectionparabolicGreen}, we  assume   that $\mathcal{H}$ is virtually abelian of rank $d$ and that $\mu$ is spectrally degenerate along $\mathcal{H}$. According to \cite[Lemma~4.16]{DGstability}, it implies that $R_{\eta}=1$ for any $\eta \geq 0$,  i.e.\ $\mu$ is spectrally degenerate along $\mathcal{N}_\eta(\mathcal{H})$.

\medskip

Our first goal is to obtain asymptotics of the $\big(\lceil d/2\rceil-1\big)$th derivative of the Green function $G_{\eta,  r}$ at 1; see Proposition~\ref{propestimatesderivativesparabolicGreen} below.

\medskip

We fix $\alpha \in (0, 1)$ and consider  the transition kernel $\tilde{p}_{\eta,  r}$ defined by
$$\tilde{p}_{\eta,  r}(x,  y)=\alpha \delta_{x,  y}+(1-\alpha)p_{\eta,  r}(x,  y).$$
Let  $\tilde{G}_{\eta,  r}$ be  the corresponding Green function.
Then,  by \cite[Lemma~9.2]{Woess-book}, 
$$\tilde{G}_{\eta,  r}(e,  e|t)=\frac{1}{1-\alpha t}G_{\eta,  r}\bigg (e,  e\bigg|\frac{(1-\alpha)t}{1-\alpha t}\bigg ) .$$
Hence,  up to a constant that only depends on $\alpha$ and $j$,  the $j$th derivative of $\tilde{G}_{\eta,  r}$ and  $G_{\eta,  r}$  coincide    at 1.
Therefore,  up to replacing $p_{\eta,  r}$ by $\tilde{p}_{\eta,  r}$ we can assume that,  $p_{\eta,  r}(x,  x)>0$ for every $x$ so that  the transition kernel $p_{\eta,  r}$ is aperiodic. {\bf We keep this assumption for all this section.}

\medskip
By definition,  there exists a subgroup of $\mathcal{H}$  of finite index which is isomorphic to $\Z^{d}$.
Any section $\mathcal{H}/\Z^{d}\to \mathcal{H}$ allows us to identify $\mathcal{H}$ with $\Z^{d}\times F$ for some finite set $F$.
As in \cite{DGstability} and \cite{DGGP}, the group  $\Gamma$ can be $\mathcal{H}$-equivariantly identified with $\mathcal{H}\times \N$.
Indeed,  the parabolic subgroup $\mathcal{H}$ acts by left multiplication on $\Gamma$ and the quotient is countable.
We order elements in the quotient according to their distance to $\mathcal{H}$.
It follows that
\begin{enumerate}
    \item $\mathcal{N}_{\eta}(\mathcal{H})$ can be $\Z^{d}$-equivariantly identified with $\Z^{d}\times \{1, ...,  N_\eta\}$, 
    \item if $\eta\leq \eta'$,  then $N_\eta\leq N_{\eta'}$. In other words,  the set $\Z^{d}\times \{1, ...,  N_\eta\}$,  identified with $\mathcal{N}_{\eta}(\mathcal{H})$,  is a subset of $\Z^{d}\times \{1, ...,  N_{\eta'}\}$, identified with $\mathcal{N}_{\eta'}(\mathcal{H})$.
\end{enumerate}
Each  element of $\mathcal{N}_\eta(\mathcal{H})$ can be written  as $(x,  j)$,  where $x\in \Z^{d},  j\in \{1, ...,  N_\eta\}$.
We also write $p_{j,  j';r}(x,  x')=p_{\eta,  r}((x,  j), (x',  j'))$ for simplicity.

\begin{definition}
Letting $u\in \R^{d}$.. The \emph{vertical displacement transition matrix} $F_r(u)$ is defined by:   for all $j,  j'\in \{1, ...,  N_\eta\}$,  
its $(j,  j')$ entry   equals
$$F_{j,  j';r}(u):=\sum_{x\in \Z^{d}}p_{j,  j';r}(0, x)\mathrm{e}^{u\cdot x}.$$
\end{definition}
This transition matrix  was introduced in \cite{Dussaule} for $r=1$.
Many properties are derived there from the fact that it is strongly irreducible,  i.e. \ there exists $n$ such that every entry of $F_{r}(u)^n$ is positive.
Since by \cite[Lemma~3.2]{Dussaule}, 
 $$F_{j,  j';r}(u)^n=\sum_{x\in \mathbb Z^d}p_{j,  j' ; r}^{(n)}(0,  x)\mathrm{e}^{u\cdot x}, $$
where $p_{j,  j' ; r}^{(n)}(0,  x) = p^{(n)}_{\eta, r}((0,j), (x, j'))$, strong irreducibility is deduced from the fact that $p_{\eta,  r}$ is aperiodic.
Denote by $\mathcal{F}_r\subset \R^{d}$ the interior of the set of $u\in \mathbb R^d$ where $F_r(u)$ has finite entries.
By the Perron-Frobenius Theorem \cite[Theorem~1.1]{Seneta}, the matrix  $F_r(u)$ has a positive dominant eigenvalue $\lambda_r(u)$ on $\mathcal{F}_r$.
Also,  by \cite[Proposition~3.5]{Dussaule},  the function $\lambda_r$ is continuous and strictly convex on $\mathcal{F}_r$ and reaches its minimum at some value $u_r$.
Moreover,  by Lemma~\ref{exponentialmoments}, uniformly in $r$,  the transition kernel $p_{\eta,r}$ has arbitrary large exponential moments,  up to taking $\eta$ large enough.

\medskip

Denote by $B(0,  M)$ the  closed ball of radius $M$ and center $0$ in $\R^{d}$.
It then follows from \cite[(5),  Proposition~4.6]{DGGP} that for large enough $\eta$,  there exists a constant $M$ such that, for every $u\in B(0,  M)$,  the matrix $F_r(u)$ has finite entries and  the minimum of the function $\lambda_r$ is reached at some $u\in B(0,  M)$.
In other words,  $u_r\in B(0,  M)\subset \mathcal{F}_r$; note that $M$ is independent of $r$.

\medskip

We now fix such $\eta$ large enough   so that the size of the matrices $F_r(u)$ is a fixed number,  say $K$.
We endow $M_K(\R^{d})$ with a matrix norm.
For fixed $r$,  the function $F_r$ is continuous in $u$.
We then endow the space of  continuous functions from $B(0,  M)$ to $M_K(\R^{d})$ with the norm $\|\cdot \|_\infty$.
We also choose an arbitrary norm on $\R$ and endow the space of  continuous functions from $B(0,  M)$ to $\R$ with the norm $\|\cdot \|_\infty$.
According to \cite[Lemma~5.4,  Lemma~5.5]{DGstability},  the functions $r\mapsto F_r$ and $r\mapsto \lambda_r$ are continuous for these norms.


\subsection{Differentiability of the parabolic spectral radius}
By \cite[Theorem~8.23]{Woess-book},  the spectral radius $\rho_{\eta}(r)=R_{\eta}(r)^{-1}$ satisfies
\begin{equation}\label{equationrholambda}
\rho_{\eta}(r)=\inf \lambda_r(u)=\lambda_r(u_r).
\end{equation}
Actually,  \cite[Theorem~8.23]{Woess-book} only deals with finitely supported transition kernels on $\Z^d\times \{1, ...,  N\}$,  but the statement remains valid for the transition kernel $p_{\eta,  r}$ since this condition of finite support can be dropped,  see \cite[(8.24)]{Woess-book}.
In what follows,  we need to take the derivative of the fonction $r\mapsto \rho_{\eta}(r)$,  hence we  first  prove that this function is differentiable.

\begin{lemma}
For every $ v\in \mathcal N_\eta(\mathcal H)$,  the function $r\mapsto p_{\eta,  r}(e,  v)$ is continuously differentiable on $[0,  R_\mu]$.
\end{lemma}

\begin{proof}
Recall that $p_{\eta,  r}(e,  v)=G(e,  v;\mathcal{N}_\eta(\mathcal{H})^c|r)$,  so it can be expressed as a power series in $r$ with positive coefficients $a_n$.
These coefficients are at most equal to $\mu^{*n}(v)$.
Since   the random walk on $\Gamma$ is convergent,  it follows that
$$\sum_{n\geq 0}na_nr^{n-1}\leq \sum_{n\geq 0}n\mu^{*n}(x)R_\mu^n<+\infty.$$
By monotonous convergence, the fonction $r\mapsto p_{\eta,  r}(e,  v)$ is continuously differentiable.
\end{proof}

For simplicity,  we   write $p_r=p_{\eta,  r}$ and   denote by $p'_r$ the derivative of $p_r$ at $r$; the  kernels 
  $p_r$ and $p'_r$ are both  $\Z^{d}$-invariant transition kernels on $\Z^{d}\times \{1, ...,  N_\eta\}$.
By Lemma~\ref{lem:r-derivate-Green},  it holds :

\begin{equation}\label{derivativep_r}
\begin{split}
\frac{d}{dr}rp_r(x,  y)&=p_r(x,  y)+rp'_r(x,  y)\\
&=\sum_{z\notin \mathcal{N}_\eta(\mathcal{H})}G(x,  z;\mathcal{N}_\eta(\mathcal{H})^c|r)G(z,  y;\mathcal{N}_\eta(\mathcal{H})^c|r).
\end{split}
\end{equation}
The following statement   can be thought of an enhanced version of relative Ancona inequalities. This is the first time  we use the fact that parabolic subgroups are virtually abelian.

\begin{proposition}\label{Anconahoroballs}
Let $\eta\geq 0$ be fixed.
There exists $C=C_\eta$ such that the following holds.
Let $x\in \mathcal{N}_\eta(\mathcal{H})$ and let $y\notin \mathcal{N}_\eta(\mathcal{H})$.
Consider a geodesic in $\Gamma$ from $y$ to $\mathcal{H}$ and denote by $\tilde{y}$ the point in $\Gamma$ at distance $\eta$ from $\mathcal{H}$ on this geodesic.
Then, 
$$G(x,  y;\mathcal{N}_\eta(\mathcal{H})^c|r)\leq CG(x, \tilde{y};\mathcal{N}_\eta(\mathcal{H})^c|r)G(\tilde{y},  y;\mathcal{N}_\eta(\mathcal{H})^c|r).$$
\end{proposition}

\begin{proof}
For simplicity, we have assumed that the generating set $S$ of   $\Gamma$ equals  the support of $\mu$. Let us fix $x,  y$ in $\Gamma$,  a geodesic from $y$ to $\mathcal H$ and $\tilde y$ satisfying the previous hypotheses. 

If $G(x,  y;\mathcal{N}_\eta(\mathcal{H})^c|r)=0$,  i.e.\ if there is no trajectory of the random walk from $x$ to $y$ staying outside $\mathcal{N}_\eta(\mathcal{H})$,  then there is nothing to prove.

Otherwise there exists such a trajectory from $x$ to $y$ and we denote by $\tilde{x}$ the first point on this trajectory outside $\mathcal{N}_\eta(\mathcal{H})$.

Denote by $\mathcal{H}(\tilde{x})$ the set of $h\in \mathcal{H}$ such that there is a trajectory outside $\mathcal{N}_\eta(\mathcal{H})$ from $\tilde{x}$ to $h\tilde{x}$.
Since the random walk is symmetric,  the set $\mathcal{H}(\tilde{x})$ is a subgroup of $\mathcal{H}$.
 Note that there exists finitely many connected components $\mathbf{c}_1, ..., \mathbf{c}_m$  of trajectories of the random walk lying in the  the set of points $z\notin \mathcal{N}_\eta(\mathcal{H})$ that project on $\mathcal{H}$ at $e$.
For any index $j$ such that there exists  trajectories from $\tilde{x}$ to some $h\mathbf{c}_j, h\in \mathcal{H}$, we  choose such  a particular trajectory and denote by $z_j$ its endpoint and by $h_j$ the corresponding point in $\mathcal{H}$.
Consider now a trajectory $\gamma$ starting at $\tilde{x}$ and staying outside $\mathcal{N}_\eta(\mathcal{H})$; let $z$ be its endpoint.
This point $z$ lies in some connected component $h\mathbf{c}_j$, hence  we can find a trajectory from $z$ to $hh_j^{-1}z_j$ staying outside $\mathcal{N}_\eta(\mathcal{H})$.
We now  concatenate the translated trajectory from $hh_j^{-1}z_j$ to $hh_j^{-1}\tilde{x}$.
In other words,  we  can add a trajectory of fixed length to $\gamma$ to reach some $h'\tilde{x}$.
In particular,  the endpoint of $\gamma$ projects on $\mathcal{H}$ within a bounded distance of $\mathcal{H}(\tilde{x})$.

The following notion was introduced in \cite[Definition~3.11]{DGstability}.
\begin{definition}
Let $k,  c>0$,  $A\subset \Gamma$ and $y\in \Gamma$. The set $A$ is \emph{$(k,  c)$-starlike around $y$} if for all $z\in A$,  there exists a path of length at most $k d(y,  z)+c$ staying in $A$.
\end{definition}

Denote by $\tilde{y}'$ the point on the chosen geodesic from $y$ to $\mathcal{H}$ which is at distance $\eta+1$ from $\mathcal{H}$.
We now prove the following.

\begin{lemma}
There exist positive constants $k,  c$ only depending on $\eta$ such that the connected component of $\tilde x$ in $\Gamma \backslash \mathcal{N}_\eta(\mathcal{H})$ is $(k,  c)$-starlike around $\tilde{y}'$.
\end{lemma}

\begin{proof}
We have to prove that,  for every $z\notin \mathcal{N}_\eta(\mathcal{H})$ that can be reached by a trajectory starting at $\tilde{x}$,  there exists a path of length at most $k d(z, \tilde{y}')+c$ which stays outside $ \mathcal{N}_\eta(\mathcal{H})$ and joins $\tilde{y}'$.

Let $z$ be such a point and $\tilde{z}$ be a point on a geodesic from $z$ to $\mathcal{H}$ at distance $\eta+1$ from $\mathcal{H}$.
Denote by $z_0$ and  $y_0$,  the  respective projection of $z$ and $y$ on $\mathcal{H}$. 
By  the above,  we can find a trajectory of fixed length starting at some point $h\tilde{x}$ and ending at $\tilde{z}$.
Similarly,  we can find a trajectory of fixed length starting at $\tilde{y}'$ and ending at some $h'\tilde{x}$.

\medskip

 Since $\mathcal{H}$ is 
virtually abelian,  the subgroup $\mathcal{H}(\tilde{x})$   is quasi-isometrically embedded in $\Gamma$.
Therefore, we can find a path $(h_j)_j$ from $h$ to $h'$ staying inside $\mathcal{H}(\tilde{x})$  with length at most $k_0 d(h,  h')+c_0$ for some fixed constants $k_0,  c_0$.
By concatenating successive trajectories from $h_j\tilde{x}$ to $h_{j+1}\tilde{x}$,  we can thus find a trajectory from $h\tilde{x}$ to $h'\tilde{x}$ of length at most $k_1 d(z_0,  y_0)+c_1$ staying outside $\mathcal{N}_\eta(\mathcal{H})$,  where $k_1,  c_1>0$ only depend  on $\eta$.
Hence,  there exist $c_2, k_2>0$  and  a trajectory $\gamma$ from $z$ to $\tilde{y}'$ of length at most $k_2 (d(z, \tilde{z})+d(z_0,  y_0))+c_2$ and staying outside $\mathcal{N}_\eta(\mathcal{H})$.
The distance formula \cite[Theorem~3.1]{Sisto-projections} shows that there exist  positive constants $c_3$ and $k_3$ such that
$$d(z, \tilde{y}')\geq k_3^{-1} \left (d(z,  z_0)+d(\tilde{y}',  y_0)+d(z_0,  y_0)\right )-c_3.$$
Hence,  the length of $\gamma$ is at most $k d(z, \tilde{y}')+c$,  where $k$ and $c$ only depend on $\eta$.
\end{proof}

Note that $\tilde{y}'$ is within a uniform bounded distance, depending only on $\eta$,  of a point on a relative geodesic from $\tilde{x}$ to $y$.
 From \cite[Proposition~3.12]{DGstability}, it  follows that
\begin{equation}\label{equationproofAncona}
G(\tilde{x},  y;\mathcal{N}_\eta(\mathcal{H})^c|r)\leq CG(\tilde{x}, \tilde{y}';\mathcal{N}_\eta(\mathcal{H})^c|r)G(\tilde{y}',  y;\mathcal{N}_\eta(\mathcal{H})^c|r).
\end{equation}
Finally,  the existence of  one-step paths from $\tilde{y}'$ to $\tilde{y}$ and  $x$ to $\tilde{x}$ yields
$$G(x,  y;\mathcal{N}_\eta(\mathcal{H})^c|r)\lesssim G(\tilde{x},  y;\mathcal{N}_\eta(\mathcal{H})^c|r), $$
$$G(\tilde{x}, \tilde{y}';\mathcal{N}_\eta(\mathcal{H})^c|r)\lesssim G(x, \tilde{y};\mathcal{N}_\eta(\mathcal{H})^c|r)$$
and
$$G(\tilde{y}',  y;\mathcal{N}_\eta(\mathcal{H})^c|r)\lesssim G(\tilde{y},  y;\mathcal{N}_\eta(\mathcal{H})^c|r).$$
Therefore,  one can replace $\tilde{x}$ by $x$ and $\tilde{y}'$ by $\tilde{y}$ in~(\ref{equationproofAncona}). The   proof of Proposition~\ref{Anconahoroballs} is complete.
\end{proof}

\begin{proposition}\label{exponentialmomentsderivative}
Let $M\geq 0$.
There exists $\eta_M$ such that for every $\eta\geq \eta_M$ and for every $r\leq R_\mu$, 
the transition kernel $p'_r$ has exponential moments up to $M$.
\end{proposition}

\begin{proof}
In view of~(\ref{derivativep_r}),  it is enough to prove that for large enough $\eta$, 
$$\sum_{y\notin \mathcal{N}_\eta(\mathcal{H})}G(e,  y;\mathcal{N}_\eta(\mathcal{H})^c|r)G(y,  v;\mathcal{N}_\eta(\mathcal{H})^c|r)
\leq C\mathrm{e}^{-2M\|x_0\|}, $$
where $x_0$ is the projection of $x$ on $\Z^{d_k}$.
For every $y\notin \mathcal{N}_\eta(\mathcal{H})$,  denote by $y_0$ its projection on $\Z^{d}$
and let $\tilde{y}$ be the point at distance $\eta$ from $\mathcal{H}$ on a geodesic from $y$ to $\mathcal{H}$.
Also,  let $P_{y_0}$ be the set of points in $\Gamma$ whose project onto $\Z^{d}$ at $y_0$.
Proposition~\ref{Anconahoroballs} shows that
\begin{align*}
&\sum_{y\notin \mathcal{N}_\eta(\mathcal{H})}G(e,  y;\mathcal{N}_\eta(\mathcal{H})^c|r)G(y,  v;\mathcal{N}_\eta(\mathcal{H})^c|r)
\\ &\hspace{1cm}\lesssim \sum_{y_0\in \Z^{d}}\sum_{y\in P_{y_0}}G(e, \tilde{y};\mathcal{N}_\eta(\mathcal{H})^c|r)G(\tilde{y},  v;\mathcal{N}_\eta(\mathcal{H})^c|r)\\
&\hspace{6cm}G(\tilde{y},  y;\mathcal{N}_\eta(\mathcal{H})^c|r)G(y, \tilde{y};\mathcal{N}_\eta(\mathcal{H})^c|r).
\end{align*}
Since the $\mu$-random walk   is   convergent, equality~(\ref{derivativep_r}) yields
$$\sum_{y\in P_{y_0}}G(\tilde{y},  y;\mathcal{N}_\eta(\mathcal{H})^c|r)G(y, \tilde{y};\mathcal{N}_\eta(\mathcal{H})^c|r)\lesssim G'(e,  e|r)<+\infty.$$
Consequently, 
\begin{align*}
\sum_{y\notin \mathcal{N}_\eta(\mathcal{H})}G(e,  y;\mathcal{N}_\eta(\mathcal{H})^c|r)G(y,  v;\mathcal{N}_\eta(\mathcal{H})^c|r)
\\ \lesssim\sum_{y_0\in \Z^{d}} G(e, \tilde{y};\mathcal{N}_\eta(\mathcal{H})^c|r)G(\tilde{y},  v;\mathcal{N}_\eta(\mathcal{H})^c|r).
\end{align*}
By Lemma~\ref{exponentialmoments},  for $\eta$   large enough,  it holds 
$ G(e, \tilde{y};\mathcal{N}_\eta(\mathcal{H})^c|r)\lesssim \mathrm{e}^{-4M\|y_0\|}$ and $G\tilde{y},  v;\mathcal{N}_\eta(\mathcal{H})^c|r)\lesssim \mathrm{e}^{-2M\|y_0-x_0\|}\lesssim \mathrm{e}^{2M\|y_0\|-2M\|x_0\|}$ ; therefore, 
$$\sum_{y\notin \mathcal{N}_\eta(\mathcal{H})}G(e,  y;\mathcal{N}_\eta(\mathcal{H})^c|r)G(y,  v;\mathcal{N}_\eta(\mathcal{H})^c|r)
\lesssim \left (\sum_{y_0\in \Z^{d}} \mathrm{e}^{-2M\|y_0\|}\right ) \mathrm{e}^{-2M\|x_0\|}$$
This concludes the proof,  since 
$\sum_{y_0\in \Z^{d}} \mathrm{e}^{-2M\|y_0\|}<+\infty.$  
\end{proof}
 Proposition~\ref{exponentialmomentsderivative} now allows us to describe the regularity of the map $(r,  u)\mapsto F_r(u)$  on $(0,  R_\mu)\times \mathring{B}(0,  M)$ where $M$ is the constant which appears at the end of  the previous subsection and $\mathring{B}(0,  M)$ is the open ball with center $0$ and radius $M$.
 We fix $\eta$ such that both $p_r$ and $p'_r$ have exponential moments up to $M$.
 
\begin{lemma}
The function $(r,  u)\mapsto F_r(u)$ is continuously differentiable on the open set $(0,  R_\mu)\times \mathring{B}(0,  M)$.
\end{lemma}

\begin{proof}
It suffices to prove that for every $j,  j'$,  $F_{j,  j';r}(u)$ is continuously differentiable.
By definition, 
$$F_{j,  j';r}(u)=\sum_{x\in \Z^d}p_{j,  j';r}(0,  x)\mathrm{e}^{u\cdot x}.$$
 For all $x\in \mathbb Z^d$, the function $f:(r,  u)\mapsto p_{j,  j';r}(0,  x)\mathrm{e}^{u\cdot x}$ is continuously differentiable and its derivative is given by
$$\nabla_{r,  u}f(r, u)=p'_{j,  j';r}(0,  x)\mathrm{e}^{u\cdot x}v_r+p_{j,  j';r}(0,  x)\mathrm{e}^{u\cdot x}v_u(x), $$
where $v_r=(1, 0, ...0)$ and $v_u(x)=(0, x)$.
Then, 
$$\|\nabla_{r,  u}f\|_\infty\leq \sup_rp'_{j,  j';r}(0,  x)\mathrm{e}^{M'\|x\|}+\|x\|\sup_rp'_{j,  j';r}(0,  x)\mathrm{e}^{M'\|x\|}.$$
Lemma~\ref{exponentialmoments} and Proposition~\ref{exponentialmomentsderivative} show that this quantity is summable.
Hence,  by dominated convergence,  the function $(r,  u)\mapsto F_{j,  j';r}(u)$ is continuously differentiable and its derivative equals
$$\nabla_{r,  u}F_{j,  j';r}(r, u)=\sum_{x\in \Z^d}p'_{j,  j';r}(0,  x)\mathrm{e}^{u\cdot x}v_r+p_{j,  j';r}(0,  x)\mathrm{e}^{u\cdot x}v_u(x), $$
which concludes the proof.
\end{proof}
 We can now prove that the function $r\mapsto \rho_{\eta}(r)$ is differentiable on $(0, R_\mu)$ and compute the value of its derivate.
\begin{proposition}\label{spectralradiusC1}
The function $r\mapsto \rho_{\eta}(r)$ is continuously differentiable and its derivative is given by
$\rho'_{\eta}(r)=\lambda'_r(u_r)$.
\end{proposition}

\begin{proof}
The function $F\mapsto \lambda$ is analytic on the set where $F$ has a unique dominant eigenvalue.
Thus,  $(r,  u)\mapsto \lambda_r(u)$ is continuously differentiable on $(0,  R_\mu)\times \mathring{B}(0,  M)$. Moreover,  it follows from \cite[Proposition~3.5]{Dussaule} that for all $r$,  the Hessian of the map $u \mapsto \lambda_r(u)$ is positive definite. Therefore,  the implicit function theorem shows that  the function $r\mapsto u_r$ is continuously differentiable on  $(0, R_\mu)$,  and so is the function $r\mapsto \lambda_r(u_r)$.
Moreover, 
$$\rho'_{\eta}(r)=\nabla_u\lambda_r(u_r)\cdot u_r' + \lambda'_r(u_r).$$
Since $\lambda_r$ is stricly convex and reaches its minimum at $u_r$,  we have
$\nabla_u\lambda_r(u_r)=0$,  hence the desired formula.
\end{proof}

We can now extend by continuity the function $\rho'_{\eta}(r)$ on $[0,R_\mu]$, so that $r\mapsto \rho_{\eta}(r)$ is differentiable on the closed set $[0,R_\mu]$ and its derivative is given by $\lambda'_r(u_r)$.

\begin{lemma}\label{derivativespectralradiusnonzero}
For any $r\leq R_\mu$,  we have $\rho'_{\eta}(r)\neq 0$.
\end{lemma}

\begin{proof}
We just need to show that $\lambda'_r(u)\neq 0$ for any $u\in \mathring{B}(0,  M)$.
For a strongly irreducible matrix $F$,  denote by $C$ and by $\nu$ right and left eigenvectors associated to the dominant eigenvalue $\lambda$.
By the Perron-Frobenius Theorem \cite[Theorem~1.1]{Seneta},  they both have positive coefficients.
Moreover,  one can normalize them such that we have $\nu\cdot C=1$ and such that $F\mapsto C$ and $F\mapsto \nu$ are analytic functions,  see \cite[Lemma~3.3]{Dussaule}.
In particular,  denoting by $C_r(u)$ and $\nu_r(u)$ right and left eigenvectors of $F_r(u)$ associated with the eigenvalue $\lambda_r(u)$,  we get that the maps $(r,  u)\mapsto C_r(u)$ and $(r,  u)\mapsto \nu_r(u)$ are continuously differentiable and satisfy $\nu_r(u)\cdot C_r(u)=1$.
Therefore, 
$$\lambda_r(u)=\nu_r(u)\cdot F_r(u)\cdot C_r(u)$$
and so
$$\lambda'_r(u)=\lambda_r(u)\bigg(\nu_r'(u)\cdot  C_r(u)+\nu_r(u)\cdot C_r'(u)\bigg) + \nu_r(u)\cdot F_r'(u)\cdot C_r(u).$$
Differentiating in $r$ the expression $\nu_r(u)\cdot C_r(u)=1$,  we get
$$\nu_r'(u)\cdot  C_r(u)+\nu_r(u)\cdot C_r'(u)=0$$ and so
$$\lambda'_r(u)=\nu_r(u)\cdot F_r'(u)\cdot C_r(u).$$
Since $p_r'(e,  v)$ is non-negative for every $ v\in \Gamma $,  the matrix $F_r'$ has non-negative entries.
Also,  it cannot be equal to the null matrix since $p_r'(e,  v)$ is positive for at least some $v$.
Moreover,  $C_r(u)$ and $\nu_r(u)$ both have positive entries.
Hence,  $\lambda_r'(u)$ is positive.
\end{proof}


\subsection{Asymptotics of $G^{(j)}_{r}$}
We set $j=\lceil d/2\rceil -1$.
Applying the previous results,  we show  the following statement; this  is a crucial step   in the proof of our main theorem.

\begin{proposition}\label{propestimatesderivativesparabolicGreen}
Assume that $\mu$ is spectrally degenerate along $\mathcal{H}$.
If $\eta$ is large enough,  then the following holds. As $r\nearrow R_\mu$, 

$\bullet $ if  $d$ is even  then
$$G^{(j)}_{\eta,  r}(e,  e|1)\sim C \ \mathrm{Log}\left (\frac{1}{R_\mu-r}\right ).$$

$\bullet $ if  $d$ is odd  then  
$$G^{(j)}_{\eta,  r}(e,  e|1)\sim \frac{C}{\sqrt{R_\mu-r}}.$$

\end{proposition}

\begin{proof}
By \cite[Proposition~3.14]{Dussaule}  applied to the kernel $R_{\eta}(r)p_{\eta,  r}$, there exists a constant $C_r$ such that $C_r^{-1}R_{\eta}(r)^np_{\eta,  r}^{(n)}n^{d/2}-1$ converges to 0 as $n\to +\infty$.
Moreover,  the convergence is uniform on $r$ and the function $r\mapsto C_r$ is continuous; consequently, the quantity $C_r$ remains bounded away from 0 and infinity.
Fix $\epsilon>0$.
Assume first that $d$ is even,  so that $j=d/2-1$.
Then,  for large enough $n$,  say $n\geq n_0$,  independently of $r$,  we have
$$\left |C_r^{-1}n^{j}p_{\eta,  r}^{(n)}-n^{-1}\rho_{\eta}(r)^n\right |\leq \epsilon n^{-1}\rho_{\eta}(r)^n.$$
Consequently, 
$$\left |\sum_{n\geq n_0}\left (n^{j}p_{\eta,  r}^{(n)}-C_rn^{-1}\rho_{\eta}(r)^n\right )\right |\leq C_r\epsilon \sum_{n\geq 0}n^{-1}\rho_{\eta}(r)^n.$$
Hence, 
\begin{align*}
\left | \sum_{n\geq 0}n^{j}p_{\eta,  r}^{(n)}-C_r\sum_{n\geq 0}n^{-1}\rho_{\eta}(r)^n\right |\leq &\sum_{n\leq n_0-1}n^{j}p_{\eta,  r}^{(n)}+C_r\sum_{n\leq n_0-1}n^{-1}\rho_{\eta}(r)^n\\
&+\epsilon C_r \sum_{n\geq 0}n^{-1}\rho_{\eta}(r)^n.
\end{align*}
Note that
$$C_r \sum_{n\geq 0}n^{-1}\rho_{\eta}(r)^n=C_r\ \mathrm{Log}\left (\frac{1}{1-\rho_{\eta}(r)}\right ).$$
Since $\rho_{\eta}(r)$ converges to 1 as $r\nearrow R_\mu$,  this last quantity tends to infinity as $r$ converges to $R_\mu$.
In particular,  this proves that
$$\sum_{n\geq 0}n^{j}p_{\eta,  r}^{(n)}\underset{r\nearrow R_\mu}{\sim}C_r\ \mathrm{Log}\left (\frac{1}{1-\rho_{\eta}(r)}\right ),$$
and so
\begin{equation}\label{equationasymptoticderivativeGreen}
G^{(j)}_{\eta,  r}(e,  e|1) \underset{r\nearrow R_\mu}{\sim}C'_r\ \mathrm{Log}\left (\frac{1}{1-\rho_{\eta}(r)}\right ).
\end{equation}
By Proposition~\ref{spectralradiusC1},  there exists $\alpha \in \mathbb R$ such that $\rho_{\eta}(r)=1+\alpha (r-R_\mu)+o\left (r-R_\mu\right )$; 
  Lemma~\ref{derivativespectralradiusnonzero} yields $\alpha\neq 0$. Hence
$1-\rho_{\eta}(r)\sim \alpha (R_\mu-r)$.
Combined with~(\ref{equationasymptoticderivativeGreen}),  this concludes the proof.
The case where $d$ is odd is  treated in the same way.
\end{proof}
 

\section{Stability of the Martin boundary}\label{sec:stability}

 This section is dedicated to the proof of the \emph{stability} of the Martin boundary (see Definition~\ref{def:stable-Martin}) in the case where the $\mu$-random walk is  convergent; this had not been dealt with before. We adapt here the arguments \cite{Dussaule},  \cite{DGstability} and \cite{DGGP} in this   context.
 The central result of this section is the following one.

\begin{theorem}\label{corostrongstability}
Let $\Gamma$ be a relatively hyperbolic group with respect to virtually abelian subgroups.
Let $\mu$ be an admissible,  symmetric and finitely supported probability measure on $\Gamma$.
Then,  the Martin boundary of $(\Gamma, \mu)$ is stable.
\end{theorem}

\begin{proof}
Recall that the Martin boundary is stable if it satisfies four conditions given by Definition~\ref{def:stable-Martin}. Let us first explain why the three first conditions are already known to be satisfied.

\begin{itemize}
\item As already mentioned,  for all $x,  y\in \Gamma$,  since $\mu$ is admissible and $\Gamma$ is non amenable it follows from Guivarc'h \cite{Guivarch} that   $G(x,  y |R_\mu) <+\infty$,  i.e. conditon (1) of stability is satisfied.
\item From \cite[Theorem~1.2]{DGstability},  conditions (2) and (3) of stability are satisfied for any admissible random walk on a relatively hyperbolic group with  virtually abelian parabolic subgroups: the homeomorphism type of $\partial_{r\mu} \Gamma$ does not depend on $r \in (0,  R_\mu)$ (we denote it by $\partial_\mu \Gamma$) and there exists an equivariant surjective and continuous map $\phi_\mu : \partial_\mu \Gamma \to \partial_{R_\mu \mu} \Gamma$. 
\item Again from \cite[Theorem~1.2]{DGstability}  the   $r$-Martin boundary of the $\mu-$random walk  is identified  with the $r$-\emph{geometric boundaries} of $\Gamma$,  defined as follows.
When $r < R_\mu$,  the $r$-geometric boundary is constructed from the Bowditch boundary of $\Gamma$ where each parabolic limit point $\xi$ is replaced with the visual boundary of the corresponding parabolic subgroup.
Equivalently,  it is the Gromov boundary $\partial \hat \Gamma$ to which has been attached at each parabolic fixed point $\xi_{\mathcal P}$ fixed by $\mathcal P$ the visual boundary of $\mathcal P$.
At $r = R_\mu$,  the $r$-geometric boundary is given by the same construction, with the following change:  the parabolic limit points are replaced with the visual boundary of the corresponding parabolic subgroup only when the random walk is spectrally non degenerate along the underlying parabolic subgroup.
\end{itemize}

We are hence left with showing that the map $$(x,  y,  r)\in \Gamma\times \Gamma\cup \partial_\mu\Gamma\times (0,  R_\mu]\mapsto K(x,  y|r)$$
is continuous,  where for $\xi \in \partial_\mu(\Gamma)$,  we write $K(x, \xi |R_\mu) = K(x,  \phi_\mu(\xi) |R_\mu)$. Notice that this property is  proved  in \cite[Theorem~1.3]{DGstability} in the case where the random walk is spectrally non degenerate. By using the geometric interpretation of the $r$-Martin boundaries mentioned before,   we need to check that, for any  $x\in \Gamma$,  any  sequence $(y_n)_n$  in $ \Gamma\cup \partial_\mu\Gamma$ which converges to a point $\xi$ in the geometric boundary of a parabolic subgroup $\mathcal H$ along which the random walk is spectrally degenerate and any $(r_n)_n$ which converges to $R_\mu$, the sequence   $(K(x,  y_n|r_n))_n$ converges to $K(x, \xi|R_\mu)$.

As in \cite[Section~5]{DGGP},  we can assume without loss of generality that  $(y_n)_n$ stays in $\mathcal N_\eta(\mathcal H)$ for some $\eta>0$. This neighborhood can be identified with $\Z^d\times \{1, ...,  N\}$ as in Section~\ref{ssec:vertical-displacement-matrix} above.  Theorem~\ref{corostrongstability}   hence appears as a direct consequence of Proposition~\ref{thmstrongstability} below which yields  the convergence of $(K(x,  y_n|r_n))_n$. 
\end{proof}

\begin{proposition}\label{thmstrongstability}
Let $p$ be  a $\Z^d$-invariant transient kernel on $E=\Z^d\times \{1, ...,  N\}$.
Assume that $p$ is irreducible and aperiodic and has exponential moments.
Then,  the Martin boundary is stable and the function
$$(x,  y,  r)\in E\times E\cup \partial_pE\times (0,  R_p]\mapsto K(x,  y|r)$$
is continuous.
\end{proposition}
The proof of this theorem will rely on two lemmas. For any $v \in \mathbb R^d$,  we write $\langle v \rangle\in \mathbb Z^d$ the vector with integer entries which is closest (for the Euclidean distance) to $v$,  chosing the first in lexicographical order in case of ambiguity.

The next lemma is technical and is inspired by \cite[Lemma~3.28]{Dussaule}.
It will be used in a particular case,  where $p_n$ will be the convolution power of the transition kernel of first return to $\mathcal{N}_\eta(\mathcal{H})$,  $\beta_v$ will be the derivative of the eigenvalue $\lambda_v$,  $\alpha_v$ will be some suited factor and $\Sigma_v$ will be the Hessian matrix associated with $p_n$.
The fact that the quantity $a_n$ defined in the lemma uniformly converges to 0 will then be a consequence of \cite[Proposition~3.14]{Dussaule}.

\begin{lemma}\label{Prop1convergenceGreenfunction}
Let $p_n(x)$ be a sequence of real numbers,  depending on $x\in \Z^d$.
Let $K\subset \R^N$ be a compact set and  $v\mapsto \alpha_v$ and $v\mapsto \beta_v$ two continuous functions on $K$, with $\alpha_v \in \mathbb R$ and $\beta_v \in \mathbb R^d$. Let $\Sigma_v$ be a positive definite quadratic form on $\R^d$,  that depends continuously on $v\in K$. Define
$$a_n(x,  v, \gamma)=\left ( \frac{\|x-n\beta_v\|}{\sqrt{n}} \right )^{\gamma}\left ((2\pi n)^{\frac{d}{2}}p_n(x)\mathrm{e}^{v\cdot x}-\alpha_v\mathrm{e}^{-\frac{1}{2n}\Sigma_v(x-n\beta_v)}\right ).$$
Denote by $g(x)$ the sum over $n$ of the $p_n(x)$.
If $(a_n)_n$ converges to 0   uniformly in $x\in \Z^d$,  $v\in K$ and $\gamma\in [0, 2d]$,  then,  for $x\in \Z^d$ and for $v\in K$ such that $\beta_v\neq 0$,   it holds as $t$ tends to infinity, 
$$(2\pi t)^{\frac{d-1}{2}}\ \Sigma_v(\beta_v) \ g(\langle t\beta_v\rangle-x)\mathrm{e}^{v\cdot(\langle t\beta(v)\rangle-x)}=\alpha_v+o(1)$$
 where the term $o(1)$ is bounded by an asymptotically vanishing sequence which does not depend on $v$.
\end{lemma}

\begin{proof}
It is proved in \cite[Lemma~3.28]{Dussaule} that with the same assumptions,  \emph{assuming moreover that $\beta_v\neq 0$} for every $v\in K$,   
$$(2\pi t)^{\frac{d-1}{2}}g(\langle t\beta_v\rangle-x)\mathrm{e}^{v\cdot (\langle t\beta(v)\rangle-x)}\underset{t\to \infty}{\longrightarrow}\frac{\alpha_v}{\Sigma_v(\beta_v)},$$
 the convergence being uniform in $v$. The proof is the same as in \cite[Theorem~2.2]{NeySpitzer}.
; the key ingredient is that $\Sigma_v(\beta_v)$ is uniformly bounded from below by some constant $\beta$,  so that
\begin{equation}\label{equationproofstrongstability}
\mathrm{e}^{-y^2\Sigma_v(\beta_v)}\leq \mathrm{e}^{-y^2\beta}.
\end{equation}
The function on the right-hand side is integrable on $[0, +\infty)$,  and this estimate in turn allows one to prove uniform convergence in $v$,  using the dominated convergence theorem.

\medskip

We will use this lemma in a setting where we cannot assume anymore that $\beta_v>0$. We need therefore to prove that $\Sigma_v(\beta_v)(2\pi t)^{\frac{d-1}{2}}g(\langle t\beta_v\rangle-x)\mathrm{e}^{\langle t\beta(v)\rangle-x}$ converges uniformly in $v$.
We replace~(\ref{equationproofstrongstability}) by
$$\Sigma_v(\beta_v)\mathrm{e}^{-y^2\Sigma_v(\beta_v)}\leq C\frac{1}{1+y^2}.$$
The right-hand side is again an integrable function and so we can conclude exactly like in the proof of \cite[Theorem~2.2]{NeySpitzer}.
\end{proof}

This result   yields the necessary asymptotics of the Green function for every $r,  u$ such that $\nabla \lambda_r(u)\neq 0$.
When $\nabla \lambda_r(u)=0$, we need  the following lemma.

\begin{lemma}\label{Prop2convergenceGreenfunction}
With the same notations as in Proposition~\ref{Prop1convergenceGreenfunction}, assume that for every $\gamma\in [0,d-1]$, $a_n$ converges to 0, uniformly in $x\in \Z^d$ and $v\in K$.
Then, for $x\in \Z^d$ and for $v\in K$ such that $\beta_v=0$, as $y$ tends to infinity, we have
$$g(y-x)\mathrm{e}^{v\cdot(y-x)}\sim \alpha_vC_d\frac{1}{\|\Sigma(y)\|^{\frac{d-2}{2}}},$$
where $C_d$ only depends on the rank $d$.
\end{lemma}

\begin{proof}
Define
$$\tilde{g}(y)=\sum_{n\geq 1}\frac{1}{(2\pi n)^{d/2}}\alpha_v\mathrm{e}^{-\frac{1}{2n}\Sigma_v(y)}.$$
Setting $t_n=\frac{n}{\Sigma_v(y-x)}$, we have
$\Delta_n:=t_n-t_{n-1}=\frac{1}{\Sigma_v(y-x)}$ which tends to 0 as $y$ tends to infinity.
Hence,
$$\frac{1}{\alpha_v}(2\pi)^{d/2}\Sigma_v(y-x)^{\frac{d-2}{2}}\tilde{g}(y-x)=\sum_{n\geq 1}t_n^{-d/2}\mathrm{e}^{-\frac{1}{2t_n}}\Delta_n$$
is a Riemannian sum of $\int_0^{+\infty}t^{-d/2}\mathrm{e}^{-\frac{1}{2t}}dt=C'_d$.
Consequently, we just need to show that
$$\frac{g(y-x)e^{v\cdot (y-x)}}{\tilde{g}(y-x)}\underset{y\to \infty}{\longrightarrow}1.$$
Equivalently, we prove that
$$g(y-x)e^{v\cdot (y-x)}-\tilde{g}(y-x)=o\left (\|y\|^{-(d-2)}\right ).$$
We set
$$\alpha_n=\sup_{y\in \Z^d}\sup_{\gamma\in [0,d-1]} \left( \frac{\|y-x\|}{\sqrt{n}}\right )^{\gamma}\left |(2n\pi)^{d/2}p_n(y-x)\mathrm{e}^{v\cdot (y-x)}-\alpha_v\mathrm{e}^{-\frac{1}{2n}\Sigma_v(y-x)}\right |.$$
By assumption, $(\alpha_n)_n$ converges to 0 as $n$ tends to infinity.
Let $\epsilon>0$.
Then for $n\geq n_0$, $\alpha_n\leq \epsilon$.
We have
\begin{align*}
&\|y-x\|^{d-2}\left |g(y-x)e^{v\cdot (y-x)}-\tilde{g}(y-x)\right |\leq \frac{1}{\|y-x\|(2\pi)^{d/2}}\sum_{n=1}^{n_0-1}\frac{\alpha_n}{n^{1/2}}\\
&\hspace{2cm}+\frac{1}{\|y-x\|(2\pi)^{d/2}}\sum_{n=n_0}^{\|y-x\|^2}\frac{\alpha_n}{n^{1/2}}+\frac{\|y-x\|^{d-2}}{(2\pi)^{d/2}}\sum_{n>\|y-x\|^2}\frac{\alpha_n}{n^{d/2}}.
\end{align*}
The first term in the right-hand side converges to 0 as $\|y\|$ tends to infinity.
The second term is bounded by
$$\frac{\epsilon}{\|y-x\|(2\pi)^{d/2}}\int_0^{\|y-x\|^2}t^{-1/2}dt\lesssim \epsilon.$$
The last term is bounded by
$$\frac{\epsilon\|y-x\|^{d-2}}{(2\pi)^{d/2}}\int_{\|y-x\|^2}^{+\infty}t^{-d/2}dt\lesssim \epsilon.$$
This concludes the proof.
\end{proof}

 All the ingredients are gathered to achieve the demonstration of Proposition~\ref{thmstrongstability}. Lemma~\ref{Prop1convergenceGreenfunction} yields the convergence of the Martin kernels,  as $r$ tends to $R_\mu$,  $r<R_\mu$,  and $y$ tends to a point $\xi$ in the Martin boundary. Lemma~\ref{Prop2convergenceGreenfunction}  gives in turn the convergence of the Martin kernels for   $r=R_\mu$ and $y$ converging to $\xi$.
Since these two limits coincide, this  yield s the continuity of the map $(x,  y,  r)\mapsto K(x,  y|r)$.}

\begin{proof}[Proof of Proposition~\ref{thmstrongstability}]
For all $u\in \mathbb R^d$ and $r\in [0, R_\mu]$,  let $F_r(u)$ be the vertical displacement transition matrix defined in Section~\ref{ssec:vertical-displacement-matrix} and let $\lambda(r,u)$ be its dominant eigenvalue.
Let $K$ be the set of pairs $(r,u)$ such that $r\in [0, R_\mu]$ and $u$ satisfies that $\lambda(u)=1$. Since $r\mapsto \lambda_r$ is a continuous function from $[0,R_\mu]$ to the set of continuous functions of $u$, the set $K$ is compact.

Recall that for fix $k,j\in \{1,...,N\}$, we write $p_{k,j}(x,y)=p((x,k),(y,j))$, for every $x,y\in \Z^d$.
According to \cite[Proposition~3.14, Proposition~3.16]{Dussaule}, for every $k,j$ in $\{1,...,N\}$, the kernel $p_{k,j}$ satisfies the assumptions of Lemma~\ref{Prop1convergenceGreenfunction} and Lemma~\ref{Prop2convergenceGreenfunction} if we define 
$$\alpha_{(r,u)}:=\frac{1}{\mathrm{det}(\Sigma_{u})}C_r(u)_k\nu_r(u)_j,$$
$\beta_{u,r}=\nabla \lambda_r(u)$
and $\Sigma_{(r,u)}$ to be the inverse of the quadratic form associated to the Hessian of the eigenvalue $\lambda(u)$ and where $C_r(u)$ and $\nu_r(u)$ are right and left eigenvectors associated to $\lambda_r(u)$.

Consider $\xi$ in the boundary $\partial \mathcal{H}$ of $\mathcal{H}$ and let $((y_n,j_n))_n$ a sequence in   $E$ which converges to $\xi$, i.e.\ $(y_n)_n$ tends to infinity and $(y_n/\|y_n\|)_n$ converges to $\xi$.
Also write $e=(0,k_0)$ for the basepoint of $E$.

Assuming that $r<R_\mu$, \cite[Lemma~3.24]{Dussaule} shows that the map
$$u\in \{u,\lambda_{r}(u)=1\}\mapsto \frac{\nabla \lambda_r(u)}{\|\nabla\lambda_r(u)\|}$$
is a homeomorphism between $\{u,\lambda_{r}(u)=1\}$ and $\mathbb{S}^{d-1}$.
Thus, there exists $u_{r,n}$ such that
$$\frac{y_n}{\|y_n\|}=\frac{\nabla \lambda_r(u_{r,n})}{\|\nabla\lambda_r(u_{r,n})\|}.$$
As explained in the previous section, \cite[(5), Proposition~4.6]{DGGP} shows that the set $\{u,\lambda_{r}(u)=1\}$ is contained in a fixed ball $B(0,M)$.
Therefore, $u_{r,n}$ is bounded and so is $\|\nabla \lambda_r(u_{r,n})\|$.
Setting $t_n=\|y_n\|/\|\nabla\lambda_r(u_{r,n})\|$,
we see that $(t_n)_n$ tends to infinity and that
$$y_n=t_n\nabla \lambda_r(u_{r,n}).$$

Consider now a sequence $(r_n)_n$ converging to $R_\mu$, $r_n<R_\mu$.
By Lemma~\ref{Prop1convergenceGreenfunction},
\begin{align*}
&(2\pi t_n)^{\frac{d-1}{2}}G((x,k),(y_n,j_n)|r_n)\mathrm{e}^{u_{r_n,n}\cdot (y-x)}\\
&\hspace{1cm}=\frac{1}{\Sigma_{r_n,u_{r,n}}(\nabla\lambda_r(u_{r,n}))}\left(\frac{C_{r_n}(u_{r_n,n})_k\nu_{r_n}(u_{r_n,n})_{j_n}}{\mathrm{det}(\Sigma_{r_n,u_{r_n,n}})}+o(1)\right)
\end{align*}
and
\begin{align*}
&(2\pi t_n)^{\frac{d-1}{2}}G(e,(y_n,j_n)|r_n)\mathrm{e}^{u_{r_n,n}\cdot y}\\
&\hspace{1cm}=\frac{1}{\Sigma_{r_n,u_{r,n}}(\nabla\lambda_r(u_{r,n}))}\left(\frac{C_{r_n}(u_{r_n,n})_{k_0}\nu_{r_n}(u_{r_n,n})_{j_n}}{\mathrm{det}(\Sigma_{r_n,u_{r_n,n}})}+o(1)\right)
\end{align*}
Consequently,
$$K((x,k),(y_n,j_n)|r_n)=\mathrm{e}^{u_{r,n}\cdot x}\frac{\frac{C_{r_n}(u_{r_n,n})_k\nu_{r_n}(u_{r_n,n})_{j_n}}{\mathrm{det}(\Sigma_{r_n,u_{r_n,n}})}+o(1)}{\frac{C_{r_n}(u_{r_n,n})_{k_0}\nu_{r_n}(u_{r_n,n})_{j_n}}{\mathrm{det}(\Sigma_{r_n,u_{r_n,n}})}+o(1)}.$$
By \cite[Lemma~5.6]{DGstability}, the sequence $(u_{r_n,n})_n$ converges to $u_{R_\mu}$ as $(y_n/\|y_n\|)_n$ tends to $\xi$ and $(r_n)_n$ tends to $R_\mu$.
Note that the limit does not depend on $\xi$.
Indeed, $u_{R_\mu}$ is a point such that $\lambda_{R_\mu}(u_{R_\mu})=1$ and by~(\ref{equationrholambda}), the minimum of $\lambda_{R_\mu}$ is 1. Since $\lambda_{R_\mu}$ is strictly convex, the point $u_{R_\mu}$ is unique.
The $o(1)$ term above is uniform in $r$, hence
$(K((x,k),(y_n,j_n)|r_n))_n$ converges to $\mathrm{e}^{u_{R_\mu,\xi}\cdot x}\frac{C_{R_\mu}(u_{R_\mu,\xi})_k}{C_{R_\mu}(u_{R_\mu,\xi})_{k_0}}$.

Assume now that $r=R_\mu$ is fixed and $(y_n)_n$ converges to $\xi$.
We apply Lemma~\ref{Prop2convergenceGreenfunction} to the same parameters $\alpha_v$, $\beta_v$, $\Sigma_v$ to deduce that
$$K((x,k),(y_n,j_n)|R_\mu)\sim \mathrm{e}^{u_{R_\mu,\xi}\cdot x}\frac{C_{R_\mu}(u_{R_\mu,\xi})_k}{C_{R_\mu}(u_{R_\mu,\xi})_{k_0}}.$$
Thus,  as $ (y_n,j_n)\to \xi$ and $r_n\to R_\mu$ with $r_n<R_\mu$,  the limit of the two sequences $(K((x,k),(y_n,j_n)|r_n))_n$ and $(K((x,k),(y_n,j_n)|R_\mu))_n$   coincide.
\end{proof}


\section{Asymptotics of the full Green function}\label{sec:symptotics-full-Green}

The purpose in this section is to show that, for a convergent random walk on a relatively hyperbolic group whose Martin boundary is stable,  the asymptotics of the (derivatives of) the full Green function are given by the asymptotics of the (derivatives of the) Green functions associated  to the first return kernels to dominant parabolic subgroups. The precise statement is given in Theorem~\ref{thmcomparingIandJ} below. This is the last crucial step before the proof of the local limit theorem.
Note that throughout this section,   we do not need to assume  that parabolic subgroups are virtually abelian.

\medskip
 
If $x\in \Gamma$,  we define $I^{(k)}_x(r)$ by
$$I^{(k)}_x(r)=\sum_{x_1, ...,  x_k\in \Gamma}G(e,  x_1|r)G(x_1,  x_2|r)...G(x_{k-1},  x_k|r)G(x_k,  x|r).$$
For $x=e$,  we write $I^{(k)}(r)=I^{(k)}_e(r)$.
These quantities are related to the derivatives of the Green function by the following result.
We inductively define
$$F_{1,  x}(r)=\frac{d}{dr}(rG_r(e,  x))$$
and
$$F_{k, x}(r)=\frac{d}{dr}(r^2F_{k-1,  x}(r)),  k\geq 2.$$

The following Lemma generalizes Lemma~\ref{lem:r-derivate-Green} and is valid for any kernel.

\begin{lemma}\label{lemmageneralformuladerivatives}\cite[Lemma~3.2]{DussauleLLT1}
For every $x\in \Gamma$ and $r\in [0,  R_{\mu}]$,   
$$F_{k,  x}(r)=k!r^{k-1}I^{(k)}_x(r).$$
\end{lemma}

As a direct consequence, it holds.
\begin{proposition}\label{prop:IandGreen}
For every $k\geq 1$,    $x\in \Gamma$ and $r\leq R_\mu$,   
$$I^{(k)}_x(r)\asymp G(e,  x|r)+G'(e,  x|r)+...+G^{(k)}(e,  x|r).$$
Moreover,  if $k$ is the smallest integer such that $I^{(k)}(R_\mu)=+\infty$  --- or equivalently such that $G^{(k)}(e,  e|R_\mu)=+\infty$ --- then, as $r\nearrow R_\mu$, 
$$I^{(k)}_x(r)\sim C G^{(k)}(e,  x|r).$$
\end{proposition}

For any   parabolic subgroup $\mathcal{H}$ of $\Gamma$ and any $\eta\geq0$,  we also set: for every $x\in \mathcal{N}_\eta(\mathcal{H})$  and $r\in [0,  R_{\mu}]$,
$$I^{(k)}_{\mathcal{H}, \eta,  x}(r)=\sum_{x_1, ...,  x_k\in \mathcal{N}_\eta(\mathcal{H})}G(e,  x_1|r)G(x_1,  x_2|r)...G(x_{k-1},  x_k|r)G(x_k,  x|r).$$
Again,  if $x=e$,  we write $I^{(k)}_{\mathcal{H}, \eta}(r)=I^{(k)}_{\mathcal{H}, \eta,  e}(r)$.
Since $G(x,  x'|r)=G_{\mathcal{H}, \eta,  r}(x,  x'|1)$ for any $x, x'\in \mathcal{N}_\eta(\mathcal{H})$,  the quantities $I^{(k)}_{\mathcal{H}, \eta,   x}(r)$ are related to the derivatives of $G_{\mathcal{H}, \eta,  r}$ at 1 by the same formulae as in Lemma~\ref{lemmageneralformuladerivatives}.

\medskip

We now fix a finite set $\{\mathcal{H}_1, ..., \mathcal{H}_N\}$ of representatives of conjugacy classes of the parabolic subgroups.
For  $\eta\geq 0$, we set
\begin{equation}\label{definitionJ}
J^{(k)}_\eta(r)=\sum_{p=1}^NI^{(k)}_{\mathcal{H}_p, \eta}(r)
\end{equation}
and  $J^{(k)}(r)=J^{(k)}_0(r), k\geq 1$.

\begin{proposition}\label{propcomparingIandJ}
Consider a finitely generated relatively hyperbolic group $\Gamma$ and a finitely supported symmetric and admissible probability measure $\mu$ on $\Gamma$.
Assume that the $\mu$-random walk is convergent,  i.e.\ $I^{(1)}(R_\mu)$ is finite.
Let $k$ be the smallest integer such that $J^{(k)}(R_\mu)$ is infinite.
Then,  the quantity $I^{(j)}(R_\mu)$ is finite for every $j< k$ and for every $\eta\geq 0$, 
$$I^{(k)}(r)\asymp J^{(k)}_\eta(r)$$
where the implicit constant only depends on $\eta$.
\end{proposition}

\begin{proof}
Clearly,  we have $J^{(j)}(r)\lesssim I^{(j)}(r)$ for every $j$.
Also,  by \cite[Lemma~5.7]{DussauleLLT1},  the sum  $I^{(j)}$ is bounded by some quantity that only depends on all the $I^{(l)}(r)$,  $l<j$ and on all the $J^{(l)}(r)$,  $l\leq j$.
Thus,  by induction,  $I^{(j)}(R_\mu)$ is finite for every $j<k$ and $I^{(k)}(r)\lesssim J^{(k)}(r)$.
Finally,  $J^{(j)}(r)\asymp J^{(j)}_\eta(r)$,  where the implicit constant only depends on $\eta$.
\end{proof}

The purpose of this section is to prove the following theorem. Its assumptions are satisfied as soon as the parabolic subgroups are virtually abelian,  according to Theorem~\ref{corostrongstability} and \cite[Proposition~4.3]{DGstability}.

\begin{theorem}\label{thmcomparingIandJ}
Consider a finitely generated relatively hyperbolic group $\Gamma$ and a finitely supported symmetric and admissible probability measure $\mu$ on $\Gamma$.
Assume that the random walk is convergent,  i.e.\ $I^{(1)}(R_\mu)$ is finite. 
For any
 parabolic subgroup $\mathcal{H}$ of $\Gamma$  such that the random walk is spectrally degenerated along $\mathcal{H}$ and any $r\leq R_\mu$, let $p_{\mathcal{H},  r}$ be the first return kernel to $\mathcal{H}$ associated with $r\mu$.

Assume that the following holds.
\begin{itemize}
    \item The Martin boundary is stable and the function
$$(x,  y,  r)\in \Gamma\times \Gamma \cup \partial_\mu\Gamma \times (0,  R_\mu]\mapsto K(x,  y|r)$$
is continuous.
 
\item The 1-Martin boundary of $(\mathcal{H},  p_{\mathcal{H},  R_\mu})$ is reduced to a point.
\end{itemize}
Let $k$ be the smallest integer such that $J^{(k)}(R_\mu)$ is infinite.
Then,  for every $\eta\geq0$,  there exists a constant $C_\eta$ such that  as $r\nearrow R_\mu$,
$$I^{(k)}_\eta(r)\sim C_\eta J^{(k)}(r).$$
\end{theorem}
The next two subsections are dedicated to the proof of this theorem.


\subsection{Asymptotics of the second derivative}

We start with showing Theorem~\ref{thmcomparingIandJ} when $k=2$,  i.e.\ $J^{(1)}(R_\mu)$ is finite and $J^{(2)}(R_\mu)$ is infinite.  We first consider the case  $\eta = 0$.

\begin{claim}\label{claimcomparingIandJk=2}
Under the assumptions of Theorem~\ref{thmcomparingIandJ},  if $k=2$,  then there exists a positive constant  $C$ such that 
$$I^{(2)}(r)\sim C J^{(2)}(r).$$
\end{claim}

\subsubsection{Step 1. $I^{(2)}(r)$ from a transfer  operator}

The purpose of this paragraph is to prove Proposition~\ref{propI^2viapsi},  which shows that $I^{(2)}(r)$ can be written as the image of a suitable \emph{transfer operator} and a remainding term which is bounded as $r\to R_\mu$. Many computations are analogous to \cite[Section~4]{DussauleLLT2};  nevertheless we  apply this operator to functions which may not be continuous and  we cannot use thermodynamical formalism as in \cite{DussauleLLT2} to control the convergence of our estimates. We  present  detailed computations.

\medskip

We write $H(x,  y|r)=G(x,  y|r)G(y,  x|r)$.
By definition $I^{(1)}(r)=\sum_{x\in \Gamma}H(e,  x|r)$.
We introduce the function $\Phi_r$ defined by
$$\Phi_r(x)=\sum_{y\in \Gamma}\frac{G(e,  y|r)G(y,  x|r)}{G(e,  x|r)}.$$
By definition,   
$$I^{(2)}(r)=\sum_{x\in \Gamma}H(e,  x|r)\Phi_r(x).$$

We fix a finite generating set $S$.
Using the automaton $\mathcal{G}$ encoding relative geodesics given by Theorem~\ref{codingrelativegeodesics}, for any $x \in \Gamma$, we   choose a relative geodesic $[e,  x]$ from $e$ to $x$.
Also,  we will write $\Omega_0=\{\mathcal{H}_1, ..., \mathcal{H}_N\}$ and $\mathcal{H}_0=S$ so that each increment of a relative geodesic is in one of the $\mathcal{H}_j$.

Let $\alpha: \mathbb Z\to \Gamma$ be a relative geodesic such that $\alpha(0)=e$ and $\Gamma_\alpha$ be the set of elements $x\in \Gamma$ such that $x_1$ and $\alpha_1$ lie in the same $\mathcal{H}_j$.
We define $\Psi_r(\alpha)$ by
$$\Psi_r(\alpha)=\sum_{y\in \Gamma_\alpha}\frac{G(\alpha_-,  y|r)G(y, \alpha_+|r)}{G(\alpha_-, \alpha_+|r)}, $$
where $\alpha_-$ and $\alpha_+$ are the left and right extremities of $\alpha$.
We prove the following.
Let $T$ be the left shift on relative geodesics,  so that $T^k\alpha, k \in \mathbb Z$,   is the relative geodesic $\alpha(k)^{-1}\alpha$.

\begin{proposition}\label{comparingphipsi}
Let $x\in \hat{S}_n$.
Then  
$$\Phi_r(x)=\sum_{k=0}^{n-1}\Psi_r\big (T^k[e,  x]\big) + O(n).$$
\end{proposition}

\begin{proof}
Write $[e,  x]=(e,  x_1,  x_2, ...,  x_n)$.

\emph{Fix $k\leq n-1$}.
Let $\Gamma_k$ be the set of elements $y$ such that the projection of $y$ on $[e,  x]$ in the relative graph $\hat{\Gamma}$ is at $x_k$.
If there are several projections,  we choose the one which is the closest to $e$.
Let $j_k$ be such that $x_k^{-1}x_{k+1}\in \mathcal{H}_{j_k}$.

Consider some $y\in \Gamma_k$,  so that $x_k^{-1}y$ projects on $T^k[e,   x]=x_k^{-1}[e,  x]$ at $e$.
Consider the sub-relative geodesic $[e,  x_k^{-1}y]$ and
write $[e,  x_k^{-1}y]=(e,  z_1, ...,  z_m)$ and assume that $z_1\notin \mathcal{H}_{j_k}$.
By \cite[Lemma~4.16]{DussauleLLT1},  any relative geodesic from $x_k^{-1}$ to $z_m=x_k^{-1}y$ passes through a point $z$ within a bounded distance of $e$.
Also,  by \cite[Lemma~1.15]{Sisto-projections},  there exists $L\geq 0$ such that if the projection in the Cayley graph of $\Gamma$ of $z_m$ on $\mathcal{H}_{j_k}$ is at distance at least $L$ from $e$,  then the relative geodesic $[e,  x_k^{-1}y]$ contains an edge in $\mathcal{H}_{j_k}$.
If this edge is not $z_1$,  this contradicts the fact that $[e,  x_k^{-1}y]$ is a relative geodesic.
Hence,  the projection of $z_m$ on $\mathcal{H}_{j_k}$ is within a bounded distance of $e$.
By \cite[Lemma~4.16]{DussauleLLT1},  we also know that any relative geodesic from $z_m$ to $x_k^{-1}x$ passes within a bounded distance of $x_k^{-1}x_{k+1}$.
Again,  if $d(e,  x_k^{-1}x_{k+1})$ is large enough,  then such a geodesic has an edge in $\mathcal{H}_{j_k}$.
By \cite[Lemma~1.13]{Sisto-projections},  the entrance point in $\mathcal{H}_{j_k}$ is within a bounded distance of the projection of $z_m$ on $\mathcal{H}_{j_k}$.
In any case,  a relative geodesic from $z_m$ to $x_k^{-1}x$ passes within a bounded distance of $e$.
Consequently,  weak relative Ancona inequalities  (Proposition~\ref{weakAncona}) yield
\begin{equation}\label{anconakleqn-1}
\frac{G(e,  y|r)G(y,  x|r)}{G(e,  x|r)}\lesssim H(x_k,  y|r).
\end{equation}
\emph{Now,  if $y\in \Gamma_n$},  i.e.\ $y$ projects on $[e,  x]$ at $x$,  then by \cite[Lemma~4.16]{DussauleLLT1},  any relative geodesic from $e$ to $y$ passes within a bounded distance of $x$.
Hence,  weak relative Ancona inequalities yield
\begin{equation}\label{anconak=n}
\frac{G(e,  y|r)G(y,  x|r)}{G(e,  x|r)}\lesssim H(x,  y|r).
\end{equation}
Combining~(\ref{anconakleqn-1}) and (\ref{anconak=n}) yields 
$$\left |\Phi_r(x)-\sum_{k=0}^{n-1}\Psi_r\big (T^k[e,  v]\big)\right |\lesssim \sum_{k=0}^n\sum_{y\in \Gamma}H(e,  y|r)\lesssim n, $$
which is the desired bound.
\end{proof}

We deduce the following.
\begin{proposition}\label{propI^2viapsi}
We have
$$I^{(2)}(r)=\sum_{n\geq0}\sum_{k=0}^{n-1}\sum_{x\in \hat{S}_n}H(e,  x|r)\Psi_r\big (T^k [e,  x]\big)+O(1).$$
\end{proposition}

\begin{proof}
By Proposition~\ref{comparingphipsi}, 
$$\left |I^{(2)}(r)-\sum_{n\geq0}\sum_{k=0}^{n-1}\sum_{x\in \hat{S}_n}H(e,  x|r)\Psi_r\big (T^k [e,  x]\big)\right |\lesssim \sum_{n\geq 0}n\sum_{x\in \hat{S}_n}H(e,  x|r).$$
Following \cite{DussauleLLT2}, the sum $\sum_{x\in \hat{S}_n}H(e,  x|r)$ can be written as the the value at the empty sequence of the $n$th iterate of a suited transfer operator $\mathcal{L}_r$ defined on the path-space of the automaton $\mathcal{G}$ encoding relative geodesics and  applied to a function $f$,  see \cite[Section~6.1]{DussauleLLT2} for more details.
Moreover,  by \cite[Lemma~4.3]{DussauleLLT2},  the Markov shift associated with $\mathcal{G}$ has finitely many images and by \cite[Lemma~4.5,  Lemma~4.7]{DussauleLLT2},  the transfer operator $\mathcal{L}_r$ has finite pressure and is semisimple.
Thus,  by \cite[Theorem~3.5]{DussauleLLT2},  it holds  $\sum_{x\in \hat{S}_n}H(e,  x|r)\sim C\mathrm{e}^{P(r)}$,  where $P(r)$ is the maximal pressure of $\mathcal{L}_r$.
Since the random walk is convergent,  we necessarily have $P(R_\mu)<0$.
Therefore, the sum  $\sum_{n\geq 0}n\sum_{x\in \hat{S}_n}H(e,  x|R_\mu)$ is finite.
This concludes the proof,  since
$H(e,  x|r)\leq H(e,  x|R_\mu)$.
\end{proof}

\subsubsection{Step 2. $J^{(2)}(r)$ from the transfer  operator}

By Proposition~\ref{propI^2viapsi},  the Claim~\ref{claimcomparingIandJk=2} is a direct consequence of the following statement.

\begin{proposition} \label{propJ^2viapsi}
Under the previous notations, 
$$\sum_{n\geq0}\sum_{k=0}^{n-1}\sum_{x\in \hat{S}_n}H(e,  x|r)\Psi_r\big (T^k [e,  x]\big)\sim CJ^{(2)}(r).$$
\end{proposition}

\begin{proof}
We decompose  $x\in \hat{S}_n$ as $x=x_2hx_1$,  where $x_1\in \hat{S}_{n-k-1}$,  $h\in \mathcal{H}_j$ for some $j$ and $x_2\in \hat{S}_k$.
If $y$ is fixed,  we write $X_{y}^j$ (resp. $X^y_j$)  for the set of elements $z$ of relative length $j$ that can precede (resp. follow) $y$ in the automaton $\mathcal{G}$.
We also write $X_y$,  respectively $X^y$ for the set of all elements $z$ that precede,  respectively follow $y$.
We thus need to study
\begin{align*}
& \widetilde{\sum}:= \sum_{k\geq 0}\sum_{n\geq k+1}\sum_{x_1\in \hat{S}_{n-k-1}}\sum_{h\in X_{x_1}^1}\sum_{x_2\in X_h^k}H(e,  x_1|r)\frac{H(e,  x_2hx_1|r)}{H(e,  x_1|r)}\\
&\hspace{4cm}\sum_{h'\in \Gamma_h}\sum_{x'\in X^{h'}}\frac{G(x_2^{-1},  h'x'|r)G(h'x',  hx_1|r)}{G(x_2^{-1},  hx_1|r)}.  
\end{align*}
We reorganize the sum over $h,  x_2,  h',  x'$ as
\begin{equation*}\label{I^2withchi}
\begin{split}
\sum_{h\in X_{x_1}^1}\sum_{h'\in \Gamma_h}G(e,  h'|r)G(h',  h|r)G(h,  e|r)\sum_{x_2\in X_h^k}\sum_{x'\in X^{h'}}&H(e,  x_2|r)H(e,  x'|r)\\
&\chi_r(h,  h',  x_1,  x_2,  x'), 
\end{split}
\end{equation*}
where the function $\chi_r$ is defined by
\begin{equation}\label{defchi}
\begin{split}
&\chi_r(h,  h',  x_1,  x_2,  x')
=\frac{G(hx_1,  x_2^{-1}|r)}{G(e,  x_2^{-1}|r)G(h,  e|r)G(x_1,  e|r)}\\
&\hspace{2cm}\frac{G(x_2^{-1},  h'x'|r)}{G(x_2^{-1},  e|r)G(e,  h'|r)G(e,  x'|r)}\frac{G(h'x',  hx_1|r)}{G(x',  e|r)G(h',  h|r)G(e,  x_1|r)}
\end{split}
\end{equation}
so that 
\begin{align*}
 \widetilde{\sum}&= \sum_{k\geq 0}\sum_{n\geq 0}\sum_{x_1\in \hat{S}_{n}} H(e,  x_1|r)
 \sum_{h\in X_{x_1}^1} \sum_{h'\in \Gamma_h}G(e,  h'|r)G(h',  h|r)G(h,  e|r)
 \\
 & \qquad \qquad \qquad \sum_{x_2\in X_h^k}\sum_{x'\in X^{h'}} H(e,  x_2|r)H(e,  x'|r) 
\quad \chi_r(h,  h',  x_1,  x_2,  x')
\end{align*}

\begin{proposition}\label{propconvergencechi}
 The functions $\chi_r$ are  bounded uniformly in $r \in [0, R_\mu]$.
Moreover, 
as $r$ tends to $R_\mu$,  the family $(\chi_r)_r$ uniformly converges to $\chi_{R_\mu}$.
\end{proposition}

\begin{proof}
Let us note that  every quotient in the definition of $\chi_r$ is uniformly bounded. Hence,  according to the weak relative Ancona inequalities (Proposition~\ref{weakAncona}),  we just need to prove that each of them uniformly converges   as $r$ tends to $R_\mu$.

We start with the first term.
and fix  $\epsilon>0$.
By Proposition~\ref{Anconaavoidingball},  there exists $\eta$,  independent of $r$,  such that
$$G(hx_1,  x_2^{-1};B_{\eta}(h)^c|r)\leq \epsilon G(hx_1,  x_2^{-1}|r).$$
Hence,  for every $r\leq R_\mu$, 
$$\left |\frac{G(hx_1,  x_2^{-1}|r)}{G(e,  x_2^{-1}|r)G(h,  e|r)G(x_1,  e|r)}-\sum_{u\in B_{\eta}(e)}\frac{G(x_1,  u;B_\eta(e)^c|r)G(h u,  x_2^{-1}|r)}{G(e,  x_2^{-1}|r)G(h,  e|r)G(x_1,  e|r)}\right |\lesssim \epsilon$$
and so we just need to prove that 
$\sum_{u\in B_{\eta}(e)} \frac{G(x_1,  u;B_\eta(e)^c|r)}{ G(e,  x_2^{-1}|r) }\times 
\frac{G(h u,  x_2^{-1}|r)}{ G(h,  e|r)G(x_1,  e|r)}$ converges as $r$ tends to $R_\mu$,  uniformly in $x_1,  h,  x_2$.
We study separately the two ratios which appear in this sum.

\begin{lemma}\label{lemma_1convergencechi}
For fixed $\eta$,  the ratio $\frac{G(x_1,  u;B_\eta(e)^c|r)}{G(x_1,  e|r)}$ converges to $\frac{G(x_1,  u;B_\eta(e)^c|R_\mu)}{G(x_1,  e|R_\mu)}$,  uniformly in $x_1$ and $u\in B_{\eta}(e)$.
\end{lemma}

\begin{proof}
By finiteness of $B_\eta(e)$ is finite,  it is sufficient to prove that the convergence is uniform in $x_1$.
Since the function $(x,  y,  r)\mapsto K(x,  y|r)$ is   continuous, the ratio  $\frac{G(x_1,  u|r)}{G(x_1,  e|r)}$ uniformly converges to $\frac{G(x_1,  u|R_\mu)}{G(x_1,  e|R_\mu)}$.
Conditioning on the last passage through $B_{\eta}(e)$ before $u$,  we have 
$$G(x_1,  u|r)=G(x_1,  u;B_\eta(e)^c|r)+\sum_{v\in B_\eta(e)}G(x_1,  v|r)G(v,  u;B_\eta(e)^c|r),$$
where, as $r \to R_\mu$, 

(i)  $\frac{G(x_1,  v|r)}{G(x_1,  e|r)}$ uniformly converges to $\frac{G(x_1,  v|R_\mu)}{G(x_1,  e|R_\mu)}$;

(ii)  $G(v,  u;B_\eta(e)^c|r)$ uniformly converges to $G(v,  u;B_\eta(e)^c|R_\mu)$.
\end{proof}
 Similarly, to study the  behavior of the ratio $\frac{G(h u,  v_2^{-1}|r)}{G(e,  v_2^{-1}|r)G(h,  e|r)}$ as $r \to R_\mu$, we use   Proposition~\ref{Anconaavoidingball} and   replace $G(h u,  v_2^{-1}|r)$ by
$$\sum_{v\in B_{\eta'}(e)}G(h u,  v;B_{\eta'}(v)^c|r)G(v,  v_2^{-1}|r), $$
where $\eta'$ only depends on $\epsilon$ and $\eta$.
As above, we check   that both
$\frac{G(h u,  v;B_{\eta'}(e)^c|r)}{G(h,  e|r)}$ and $\frac{G(v,  v_2^{-1}|r)}{G(e,  v_2^{-1}|r)}$ uniformly converge,  as $r\to R_\mu$.

This shows uniform convergence of the first quotient in the definition of $\chi_r$. We deal similarly with the two other ones to conclude.
\end{proof}

Let $\epsilon>0$.
Since $\sum_{x}H(e,  v|R_\mu)$ is finite,  if $|R_\mu-r|$ is small enough,  we have
\begin{align*}
&\sum_{k\geq0}\sum_{n\geq 0}\sum_{x_1\in\hat{S}_n}H(e,  x_1|r) \sum_{h\in X_{x_1}^1} \sum_{h'\in \Gamma_h}G(e,  h')G(h',  h)G(h,  e)\\
&\hspace{0.5cm} \sum_{x_2\in X_h^k}H(e,  x_2|r)\sum_{x'\in X^{h'}} H(e,  x'|r)\big|\chi_r(h,  h',  x_1,  x_2,  x')-\chi_{R_\mu}(h,  h',  x_1,  x_2,  x')\big|\\
&\hspace{0.5cm}\leq\epsilon J^{(2)}(r).
\end{align*}

We can thus replace $\chi_r$ by $\chi_{R_\mu}$ in
the expression of $\widetilde \sum$.

\begin{proposition}\label{propconvergencechi_Rtildechi}
Consider a parabolic subgroup $\mathcal{H}$ along which the random walk is spectrally degenerate.
As $h,  h'\in \mathcal{H}$ tend to infinity and $d(h,  h')$ tends to infinity,  the function
$$\sum_{k\geq0}\sum_{x_2\in X_h^k}\sum_{x'\in X^{h'}}H(e,  x_2|r)H(e,  x'|r)\chi_{R_\mu}(h,  h',  x_1,  x_2,  x')$$
converges to a function $\tilde{\chi}_{\mathcal{H}}(x_1)$.
Moreover,  the convergence is uniform in $x_1$ and $r \in [0, R_\mu]$ and the function $   \tilde{\chi}_{\mathcal{H}} $ is bounded.
\end{proposition}
We first prove the following lemma.
\begin{lemma}\label{x_2indeph}
If $D$ is large enough and if $d(e,  h)>D$,  the set of $x_2$ that can precede $h$ lying in a fixed parabolic subgroup $\mathcal{H}$ is independent of $h$.
\end{lemma}

\begin{proof}
Let $h_1$ and $h_2$ be in the same parabolic subgroup $\mathcal{H}$,  such that both have length bigger than $D$.
Assume that $x$ can precede $h_1$.
By \cite[Lemma~4.11]{DussauleLLT1},  if $D$ is large enough,  the concatenation of $x$ and $h_2$ is a relative geodesic.
Now,  consider  elements $\tilde{x}$ and  $\tilde{h}$ in $\mathcal{H}$ such that $\tilde{x}\tilde{h}=xh_2$.
Hence $\tilde{x}\tilde{h}h_2^{-1}h_1=xh_1$ 
so that  $\tilde{x}\geq x$ in the lexicographical order.
If $\tilde{x}>x$,  then the concatenation of $\tilde{x}$ and $\tilde{h}$ is bigger than the concatenation of $x$ and $h_2$.
Otherwise,  $\tilde{x}=x$ and so $\tilde{h}=h_2$.
Therefore $x$ can precede $h_2$.
\end{proof}

The same proof does not apply to elements $x'$ that can follow $h'$.
However,   decomposing elements of $\Gamma$ as $h'x'$ and choosing  the inverse lexicographical order rather than the original lexicographical order,  we get similarly the following result.

\begin{lemma}\label{x'indeph'}
If $D$ is large enough and if $d(e,  h')>D$,  the set of $x'$ that can follow $h'$ lying in a fixed parabolic subgroup $\mathcal{H}$ is independent of $h'$.
\end{lemma}

We can now prove Proposition~\ref{propconvergencechi_Rtildechi}.

\begin{proof}
By Lemmas~\ref{x_2indeph} and~\ref{x'indeph'},  it is enough to prove that $\chi_{R_\mu}(h,  h',  x_1,  x_2,  x')$ converges to a function,  as $h,  h'$ tend to infinity and $d(h,  h')$ tends to infinity,  uniformly in $x_1,  x_2,  x'$.
Uniformity is proved using Proposition~\ref{Anconaavoidingball},  as in the proof of Proposition~\ref{propconvergencechi}.
Hence,  we just need to prove that for fixed $u$ and $v$,  the ratio $\frac{G(hu,  v|R_\mu)}{G(h,  e|R_\mu)}$ converges as $h$ tends to infinity.
We write
$$\frac{G(hu,  v|R_\mu)}{G(h,  e|R_\mu)}=\frac{G(hu,  v|R_\mu)}{G(hu,  e|R_\mu)}\frac{G(hu,  e|R_\mu)}{G(h,  e|R_\mu)}=\frac{G(hu,  v|R_\mu)}{G(hu,  e|R_\mu)}\frac{G(u,  h^{-1}|R_\mu)}{G(e,  h^{-1}|R_\mu)}.$$
Both $hu$ and $h^{-1}$ tend to infinity.
Since we assume that the Martin boundary of the first return kernel to $\mathcal{H}$ is reduced to a point, both ratios  $\frac{G(hu,  v|R_\mu)}{G(hu,  e|R_\mu)}$ and $\frac{G(u,  h^{-1}|R_\mu)}{G(e,  h^{-1}|R_\mu)}$ converge,  as $h$ tends to infinity.
\end{proof}

To simplify notations,  we set
\begin{equation}\label{defItilde}
\begin{split}
\tilde{I}^{(2)}(r)=\sum_{n\geq 0}\sum_{x_1\in \hat{S}_n}&H(e,  x_1|r)\sum_{h\in X_{x_1}^1}\sum_{h'\in \Gamma_h}G(e,  h'|r)G(h',  h|r)G(h,  e|r)\\
&\sum_{k\geq 0}\sum_{x_2\in X_h^k}\sum_{x'\in X^{h'}}H(e,  x_2|r)H(e,  x'|r)
\chi_{R_\mu}(h,  h',  x_1,  x_2,  x').
\end{split}
\end{equation}
We also write $\mathcal{H}_h$ for the parabolic subgroup containing $h$. When $h$ lies in several parabolic subgroups,  we arbitrarily choose one  of them and, if possible,  we choose one along which the random walk is spectrally degenerate.
Recall that the intersection of two parabolic subgroups is finite,  see \cite[Lemme~4.7]{DrutuSapir}; hence $\mathcal{H}_h$ is uniquely defined    if $d(e,  h)$ is large enough.
Finally,  if the random walk is  spectrally non  degenerate along $\mathcal{H}$,  we set $\tilde{\chi}_{\mathcal{H}}(x_1)=1.$

\begin{proposition}
Let $\epsilon>0$.
If $|r-R_\mu|$ is small enough,  then
\begin{align*}
 &\left |\tilde{I}^{(2)}(r)-\sum_{n\geq 0}\sum_{x_1\in \hat{S}_n}H(e,  x_1|r)\sum_{h\in X_{x_1}^1}\sum_{h'\in \Gamma_h}G(e,  h'|r)G(h',  h|r)G(h,  e|r)\tilde{\chi}_{\mathcal{H}_h}(x_1)\right |\\
&\lesssim \epsilon J^{(2)}(r).
\end{align*}
\end{proposition}

\begin{proof}
There exists $D_\epsilon$ such that if the random walk is spectrally degenerate along $\mathcal{H}_h$,  then if $d(e,  h),  d(e,  h'),  d(h,  h')\geq D_\epsilon$.
$$\left |\sum_{k\geq0}\sum_{x_2\in X_h^k}\sum_{x'\in X^{h'}}H(e,  x_2|r)H(e,  x'|r)\chi_{R_\mu}(h,  h',  x_1,  x_2,  x')-\tilde{\chi}_{\mathcal{H}_h}(x_1)\right |\leq \epsilon.$$
On the one hand, the sub-sum in~(\ref{defItilde}) over the $h$ and $h'$ such that,  either $d(e,  h)< D_\epsilon$  or $d(e,  h')<D_\epsilon$   or $d(h,  h')< D_\epsilon$, is uniformly bounded.
Hence,  it can be bounded by $\epsilon J^{(2)}(r)$ if $r$ is close enough to $R_\mu$.
On the other hand, the sub-sum over the $h$ and $h'$ such that $d(e,  h),  d(e,  h'),  d(h,  h')\geq D_\epsilon$ but the random walk is  spectrally non  degenerate along $\mathcal{H}_h$ is also uniformly bounded
 and    can thus be bounded as well by $\epsilon J^{(2)}(r)$ if $r$ is close enough to $R_\mu$.
\end{proof}

To conclude the proof,  we only need to prove that
$$\sum_{n\geq 0}\sum_{x_1\in \hat{S}_n}H(e,  x_1|r)\sum_{h\in X_{x_1}^1}\sum_{h'\in \Gamma_h}G(e,  h'|r)G(h',  h|r)G(h,  e|r)\tilde{\chi}_{\mathcal{H}_h}(x_1)\sim CJ^{(2)}(r).$$
The double  sum  over $x_1$ and $h$ that can precede $x_1$ is exactly the sum over every element of relative length $1+\hat{d}(e,  x_1)$,  so we can replace this double sum by a double sum over $h$ and $x_1$ that can follow $h$.
We thus need to prove that
$$\sum_{h}\sum_{h'\in \Gamma_h}G(e,  h'|r)G(h',  h|r)G(h,  e|r)\sum_{n\geq 0}\sum_{x_1\in X^{h}_n}H(e,  x_1|r)\tilde{\chi}_{\mathcal{H}_h}(x_1)\sim CJ^{(2)}(r).$$
By Lemma~\ref{x'indeph'},  if $d(e,  h)$ is large enough,  then the set of $x_1$ that can follow $h$ is independent of $h$.
This concludes the proof,  since for fixed $D$,  the sub-sum over the $h$ such that $d(e,  h)\leq D$ is uniformly bounded.
\end{proof}

\medskip

This concludes the proof of the Claim~\ref{claimcomparingIandJk=2}.
To complete the proof of Theorem~\ref{thmcomparingIandJ} in the case $k=2$,  we need to show the following.

\begin{proposition}\label{prop2comparingIandJk=2}
Under the assumptions of Theorem~\ref{thmcomparingIandJ},  if $k=2$,  then for every $\eta\geq 0$,  there exists  a positive constant $C_\eta$ such that
$$J^{(2)}_{\eta}(r)\sim C_\eta J^{(2)}(r).$$
\end{proposition}

Denote by $E_\eta$ the set of $x\in \Gamma$ such that $x$ is in $\mathcal{N}_\eta(\mathcal{H})$ for some $\mathcal{H}$.
Recall that for a relative geodesic $\alpha$ such that $\alpha(0)=e$, 
  the set  $\Gamma_\alpha$ contains all the elements $x\in \Gamma$ such that $x_1$ and $\alpha_1$ lie in the same $\mathcal{H}_j$.
Setting
$$\Psi^\eta_r(\alpha)=\sum_{y\in \Gamma_\alpha}\frac{G(\alpha_-,  y|r)G(y, \alpha_+|r)}{G(\alpha_-, \alpha_+|r)}1_{\alpha_-^{-1}\alpha_+\in E_\eta}1_{y\in E_\eta} $$
the proof of  Proposition~\ref{prop2comparingIandJk=2} is  exactly the same as the one of Claim~\ref{claimcomparingIandJk=2},  replacing $\Psi_r$ by $\Psi_r^{\eta}$.


\subsection{Higher derivatives}

We now consider the general case,  i.e.\ $J^{(j)}(R_\mu)$ is finite for every $j<k$ and $J^{(k)}$ is infinite.
We introduce the function $\Phi^{(k)}_r$  and $\Psi^{(k)}_r$ defined by:
for any $x \in \Gamma$, 
$$\Phi^{(k)}_r(x)=\sum_{y_1, ...,  y_{k-1}}\frac{G(e,  y_1|r)G(y_1,  y_2|r)...G(y_{k-1},  v|r)}{G(e,  v|r)}.$$
and, for any relative geodesic $\alpha$ such that $\alpha(0)=e$, 
$$\Psi^{(k)}_r(\alpha)=\sum_{y_1, ...,  y_{k-1}\in \Gamma_\alpha}\frac{G(\alpha_-,  y_1|r)G(y_1,  y_2|r)...G(y_{k-1}, \alpha_+|r)}{G(\alpha_-, \alpha_+|r)}.$$
As above, it holds 
$$I^{(k)}(r)=\sum_{x\in \Gamma}H(e,  v|r)\Phi^{(k)}_r(x).$$
 We have the following.

\begin{proposition}
There exists $D_k$ such that, for any $x\in \hat{S}_n$,
$$\Phi^{(k)}_r(x)=\sum_{j=0}^{n-1}\Psi^{(k)}_r\big (T^j[e,  v]\big) + O(n^{D_k}).$$
\end{proposition}

\begin{proof}
Like in the proof of Proposition~\ref{comparingphipsi},  we consider the set $\Gamma_l$ of elements $y$ such that the projection of $y$ on $[e,  x]$ is at $x_l$.
If all the $y_j$ do not lie in the same $\Gamma_l$,  then
$$\sum_{y_1, ...,  y_{k-1}}\frac{G(e,  y_1|r)G(y_1,  y_2|r)...G(y_{k-1},  x|r)}{G(e,  x|r)}$$
is bounded by a quantity only involving the $J^{(j)}(r)$,  $j<k$,  which are uniformly bounded.
Hence we can restrict the sum in the definition of $\Phi^{(k)}_r$ to the $y_j$ lying in the same $\Gamma_l$ and 
the remainder of the proof of Proposition~\ref{comparingphipsi} can  be reproduced to conclude.
\end{proof}
As  in Proposition~\ref{propI^2viapsi}, we  get
$$I^{(k)}(r)=\sum_{n\geq0}\sum_{k=0}^{n-1}\sum_{x\in \hat{S}_n}H(e,  x|r)\Psi_r^{(k)}\big (T^k [e,  x]\big)+O(1)$$
and conclude like in the case $k=2$.
This proves Theorem~\ref{thmcomparingIandJ}. \qed


\section{Proof of the local limit theorem}\label{sec:proofLLT}

In this section,  we   prove Theorem~\ref{maintheorem}.  
Recall that for fixed $x$ in $\Gamma$, 
$$I^{(k)}_x(r)=\sum_{y_1, ...,  y_k}G(e,  y_1|r)...G(y_k,  x|r).$$

\begin{proposition}
Under the assumptions of Theorem~\ref{thmcomparingIandJ},  for every $\eta\geq 0$ and every $x\in \Gamma$,  there exists $C_{\eta,  x}$ such that
$$I^{(k)}_x(r)\sim C_{\eta,  x}J^{(k)}_\eta(r).$$
\end{proposition}

\begin{proof}[Sketch of proof]
We use the same arguments  as in Theorem~\ref{thmcomparingIandJ}. Let us briefly outline the case  $k=2$.

We have
$$I^{(2)}_x(r)=\sum_{y,  z\in \Gamma}G(e,  y|r)G(y,  z|r)G(z,  e|r)\frac{G(z,  x|r)}{G(z,  e|r)}.$$
We introduce the function $\Phi_{r,  x}$ defined by
$$\Phi_{r,  x}(z)=\sum_{y\in \Gamma}\frac{G(e,  y|r)G(y,  z|r)}{G(e,  z|r)}\frac{G(z,  x|r)}{G(z,  e|r)}, $$
i.e.\ $\Phi_{r,  x}(z)=\Phi_r(z)\frac{G(z,  x|r)}{G(z,  e|r)}$.
We then have
$$I^{(2)}_x(r)=\sum_{z\in \Gamma}H(e,  z|r)\Phi_{r,  x}(z).$$
We can then reproduce the same proof,  replacing the function $\chi_r$ defined in~(\ref{defchi}) by the function $\chi_{r,  x}$ defined by
\begin{align*}
&\chi_{r,  v}(h,  h',  z_1,  z_2,  z')
=\frac{G(hz_1,  z_2^{-1}|r)}{G(e,  z_2^{-1}|r)G(h,  e|r)G(z_1,  e|r)}\\
&\hspace{1cm}\frac{G(z_2^{-1},  h'z'|r)}{G(z_2^{-1},  e|r)G(e,  h'|r)G(e,  z'|r)}\frac{G(h'z',  hz_1|r)}{G(z',  e|r)G(h',  h|r)G(e,  z_1|r)}\frac{G(z_2hz_1,  x|r)}{G(z_2hz_1,  e|r)}.
\end{align*}
Hence,  $\chi_{r,  x}=\chi_r\frac{G(z_2hz_1,  x|r)}{G(z_2hz_1,  e|r)}$ and  is bounded by a constant that only depends on $x$.
Since we assume that the Martin boundary is stable and the function
$$(x,  y,  r)\in \Gamma\times \Gamma\cup \partial_\mu\Gamma\times (0,  R_\mu] \mapsto K(x,  y|r)$$ is continuous,  the family  $(\chi_{r,  x})_r$ uniformly converges to $\chi_{R_\mu,  x}$,  as $r$ tends to $R_\mu$.
\end{proof}

\medskip

We now prove Theorem~\ref{maintheorem}. 
By Theorem~\ref{corostrongstability} and \cite[Proposition~4.3]{DGstability},  the assumptions of  Theorem~\ref{thmcomparingIandJ} are satisfied. Combining Proposition~\ref{propestimatesderivativesparabolicGreen} and Theorem~\ref{thmcomparingIandJ},  we get the following.

\begin{corollary}\label{coro:asymptotic-Green}
 For every $x$,  there exists $C_x>0$ such that, as $r\to R_\mu$

\begin{align*}
G^{(j)}(e,  x|r)& \quad  \sim \quad \frac{C_x}{\sqrt{R_\mu-r}}\qquad   if \ d \    is \ odd,\\
 and \qquad G^{(j)}(e,  x|r)&
  \quad  \sim \quad C_x \ \mathrm{Log}\left (\frac{1}{R_\mu-r}\right )\quad if  \  d\    is \ even.
 \end{align*}
\end{corollary}

In both cases, we deduce that $p^{(n)}(e,  x)\sim C_xR_\mu^{-n}n^{-d/2}$; let us explain this last step. The odd case is proved exactly like \cite[Theorem~9.1]{GouezelLalley}. The method is based on a \emph{Tauberian theorem} of Karamata and also applies to the even case.

Let us give a complete proof of the even case for sake of completeness.
A function $f$ is called \emph{slowly varying} if for every $\lambda>0$, the ratio  $f(\lambda x)/f(x)$ converges to 1 as $x$ tends to infinity.
Combining Corollary~\ref{coro:asymptotic-Green} and \cite[Corollary~1.7.3]{BinghamGoldieTeugels} (with to the slowly varying function $\log$) one gets
$$\sum_{k=0}^nk^jR_\mu^{k}\mu^{*k}(x)\sim C'_x\log (n).$$
Moreover,  by \cite[Corollary~9.4]{GouezelLalley}, 
\begin{equation}\label{munqn}
n^jR_\mu^{n}(\mu^{*n}(e)+\mu^{*n}(x))=q_n(x)+O\big (\mathrm{e}^{-c n}\big ), 
\end{equation}
where $c>0$ and the sequence $(q_n(x))_n$ is non-increasing.
We first deduce that
$$\sum_{k=0}^nk^jR_\mu^{k}q_k(e)\sim C_0\log (n).$$
The same proof as in \cite[Lemma~9.5]{GouezelLalley} yields
$$n^jR_\mu^{n}q_n(e)\sim C_1n^{-1}$$
and so,  using again~(\ref{munqn}), 
$$n^jR_\mu^{n}\mu^{*n}(e)\sim C_1 n^{-1}.$$
Recall that $j=d/2-1$,  so that
$$\mu^{*n}(e)\sim C_1 R_\mu^{-n}n^{-d/2}.$$
Since  $G^{(j)}(e,  e|r)+G^{(j)}(e,  x|r)\sim (C_e+C_x)\log 1/(R_\mu-r)$ as $r\to R_\mu$, 
we can again apply \cite[Corollary~1.7.3]{BinghamGoldieTeugels} to deduce that
$$\sum_{k=0}^nk^jR_\mu^{k}q_k(x)\sim C^{(0)}_x\log (n).$$
The fact that  $(q_n(x))_n$ is non-increasing readily implies
$$n^jR_\mu^{n}q_n(x)\sim C^{(1)}_xn^{-1}$$
as above; finally,  by~(\ref{munqn}), 
$$\mu^{*n}(x)\sim C^{(2)}_xR_\mu^{-n}n^{-d/2}.$$
Since $C_1$ is positive and the random walk is admissible,  the quantity $C^{(2)}_x$ is also  positive.
This concludes the proof of Theorem~\ref{maintheorem} and
Corollary~\ref{maincorollary} thus follows from \cite[Proposition~4.1]{GouezelLLT}. \qed


\section{Divergence and spectral positive recurrence}\label{sec:positive-recurrent}

In this section, we describe the relationships between divergence and spectral positive recurrence of random walks on relatively hyperbolic groups. As before, assume that $\Gamma$ is a finitely generated relatively hyperbolic group with respect to virtually abelian subgroups.
Fix a finite set $\Omega_0$ of representatives of conjugacy classes of parabolic subgroups.

Recall that a random walk on $\Gamma$ is called spectrally positive recurrent if it is divergent and has finite Green moments,  i.e.\ using the notations defined in (\ref{definitionJ}),  if $J^{(2)}(R_\mu)$ is finite. By definition, this occurs  when  $I^{(2)}_{\mathcal H}(R_\mu)<+\infty$ for all parabolic subgroup $\mathcal H\in \Omega_0$. 

\medskip
If the random walk driven by $\mu$ is  spectrally non  degenerate along the parabolic subgroup $\mathcal H$, then we automatically have $I^{(2)}_{\mathcal H}(R_\mu)<+\infty$. Since the random walk cannot be spectrally degenerate along virtually abelian parabolic subgroups of rank at most $4$ (see \cite[Proposition~6.1]{DGstability}), other behaviors only occur when parabolic subgroups have rank at least $5$. 

Note that the Local limit theorem (\ref{eq:LLT_abelien}) implies that when the rank of $\mathcal H$ is at least $7$, then $G_{\mathcal H}(e,e|r)$, $G'_{\mathcal H}(e,e|r)$ and $G''_{\mathcal H}(e,e|r)$ are finite at the inverse of the spectral radius of $G_{\mathcal H}$, which is at least $1$. It follow then from Proposition~\ref{prop:IandGreen} (or more precisely, from the analogous of Proposition~\ref{prop:IandGreen} for $\mathcal H$) that $I^{(2)}_{\mathcal H}(R_\mu)<+\infty$. We gather these observations in the following proposition.

\begin{proposition}
Let $\Gamma$ be a a finitely generated relatively hyperbolic group $\Gamma$ with respect to virtually abelian subgroups. Let $\mu$ be a finitely supported,  admissible,  symmetric probability measure on $\Gamma$. Assume that for every parabolic subgroup $\mathcal{H}$ of rank $5$ or $6$, the random walk is  spectrally non  degenerate along $\mathcal{H}$. Then the random walk is spectrally positive recurrent \emph{if and only if} it is divergent.
\end{proposition}

In a work in preparation \cite{DPT23}, we show that there exist examples of spectrally degenerate random walks on relatively hyperbolic groups with parabolic subgroups being virtually abelian of rank $d=5$ or $6$, which are \emph{divergent but not spectrally positive recurrent}.
They show exotic local limit theorems which are neither of the form (\ref{eq:LLT_abelien}) nor (\ref{eq:LLT_hyperbolic}). We construct such examples on free products of abelian groups.

\begin{remark}\label{remarquefinale}
Note that these exotic examples contradict \cite[Lemma 4.5]{CandelleroGilch}. Unfortunately the proof of this Lemma shows a subtle gap. A previous version of our article contained a similar gap, which leaded to the (wrong) conclusion that any divergent random walk is spectrally positive recurrent as soon as parabolic subgroups are virtually abelian.
\end{remark}

\bibliographystyle{plain}
\bibliography{LLT_CV}

\end{document}